\documentclass[11pt]{amsart}

\usepackage{amsmath}
\usepackage{amssymb}
\usepackage{amscd}
\usepackage{color}
\usepackage{hyperref}
\usepackage{mathrsfs}

\usepackage{xy}
\xyoption{all}

\newcommand{\nc}{\newcommand}

\nc{\exto}[1]{\stackrel{#1}{\longrightarrow}}
\nc{\dlim}{{\mathop{\lim\limits_{\longrightarrow}\,}}}
\nc{\ilim}{{\mathop{\lim\limits_{\longleftarrow}\,}}}
\nc{\hocolim}{{\mathop{\sf hocolim}\,}}
\nc{\holim}{{\mathop{\sf holim}}}
\nc{\lan}{\big\langle}
\nc{\ran}{\big\rangle}

\nc{\kk}{{\mathsf{k}}}

\nc{\C}{{\mathbb{C}}}
\nc{\HH}{{\mathbf{H}}}
\nc{\LL}{{\mathbb{L}}}
\nc{\RR}{{\mathbb{R}}}
\nc{\PP}{{\mathbb{P}}}
\nc{\OO}{{\mathbb{O}}}
\renewcommand{\SS}{{\mathbb{S}}}
\nc{\QQ}{{\mathbb{Q}}}
\nc{\ZZ}{{\mathbb{Z}}}

\nc{\CA}{{\mathcal{A}}}
\nc{\CB}{{\mathcal{B}}}
\nc{\CC}{{\mathcal{C}}}
\nc{\D}{{\mathcal{D}}}
\nc{\CE}{{\mathcal{E}}}
\nc{\CF}{{\mathcal{F}}}
\nc{\CG}{{\mathcal{G}}}
\nc{\CH}{{\mathcal{H}}}
\nc{\CL}{{\mathcal{L}}}
\nc{\CM}{{\mathcal{M}}}
\nc{\CN}{{\mathcal{N}}}
\nc{\CO}{{\mathcal{O}}}
\nc{\CQ}{{\mathcal{Q}}}
\nc{\CR}{{\mathcal{R}}}
\nc{\CS}{{\mathcal{S}}}
\nc{\CT}{{\mathscr{T}}}
\nc{\CU}{{\mathcal{U}}}
\nc{\CV}{{\mathcal{V}}}
\nc{\CW}{{\mathcal{W}}}
\nc{\CX}{{\mathcal{X}}}
\nc{\CY}{{\mathcal{Y}}}
\nc{\CMo}{{\mathcal{M}^\circ}}
\nc{\Co}{{{C}^\circ}}

\nc{\BY}{{\overline{Y}}}
\nc{\BYD}{{\overline{Y}{}^{|D|}}}
\nc{\OZ}{{\overline{Z}}}
\nc{\bg}{{\bar{g}}}

\nc{\bq}{{\mathbf{q}}}
\nc{\BC}{{\mathbf{C}}}
\nc{\BD}{{\mathbf{D}}}
\nc{\BG}{{\mathbf{G}}}
\nc{\BM}{{\mathbf{M}}}
\nc{\BP}{{\mathbf{P}}}
\nc{\BZ}{{\mathbf{Z}}}
\nc{\BPr}{{\mathsf{P}}}
\nc{\BR}{{\mathbf{R}}}
\nc{\BRO}[1]{{{\mathbf{R}}^{\circ}_{#1}}}
\nc{\BRD}[1]{{{\mathbf{R}}^{|D|}_{#1}}}
\nc{\BRP}[1]{{{\mathbf{R}}^{1}_{#1}}}
\nc{\BRTP}[1]{{{\mathbf{\tilde{R}}}{}^{1}_{#1}}}
\nc{\BS}{{\mathbf{S}}}
\nc{\BMS}{{{\mathbf{M}}^{{s}}}}
\nc{\BMSS}{{{\mathbf{M}}^{{ss}}}}
\nc{\BMZ}{{\mathbf{M}^{\circ}}}
\nc{\BCL}{{\mathbf{L}}}

\nc{\PCC}{{{}^\perp\CC}}

\nc{\Cl}{{\mathsf{Cliff}}}
\nc{\Clev}{{\mathop{\mathsf{Cliff}}^{\circ}}}

\nc{\FA}{{\mathfrak{A}}}
\nc{\FB}{{\mathfrak{B}}}

\nc{\fa}{{\mathfrak{a}}}
\nc{\fb}{{\mathfrak{b}}}
\nc{\fg}{{\mathfrak{g}}}
\nc{\fn}{{\mathfrak{n}}}
\nc{\fp}{{\mathfrak{p}}}
\nc{\FD}{{\mathfrak{D}}}
\nc{\FE}{{\mathfrak{E}}}
\nc{\FL}{{\mathfrak{L}}}
\nc{\FM}{{\mathfrak{M}}}
\nc{\FS}{{\mathsf{S}}}

\nc{\sfc}{{\mathsf{c}}}
\nc{\sfch}{{\mathsf{ch}}}
\nc{\sfh}{{\mathsf{h}}}

\nc{\SK}{{\mathsf{K}}}
\nc{\SM}{{\mathsf{M}}}
\nc{\SO}{{\mathsf{O}}}
\nc{\SQ}{{\mathsf{Q}}}
\nc{\SPV}{{\mathsf{S}^+\mathsf{V}}}
\nc{\SMV}{{\mathsf{S}^-\mathsf{V}}}
\nc{\SPMV}{{\mathsf{S}^\pm\mathsf{V}}}
\nc{\SSS}{{\mathsf{S}}}
\nc{\SX}{{S_X}}
\nc{\SY}{{S_Y}}
\nc{\phipsi}{{q}}
\nc{\eps}{\varepsilon}

\nc{\pim}{{\pi_-}}
\nc{\pip}{{\pi_+}}

\nc{\BE}{{\overline{\CE}}}
\nc{\TE}{{\tilde{\CE}}}
\nc{\TQ}{{\tilde{Q}}}
\nc{\TCF}{{\tilde{\CF}}}
\nc{\TCG}{{\tilde{\CG}}}
\nc{\TCL}{{\tilde{\CL}}}
\nc{\TF}{{\tilde{F}}}
\nc{\TW}{{\tilde{W}}}
\nc{\TCC}{{\tilde{\CC}}}
\nc{\TCX}{{\tilde{\CX}}}
\nc{\TCY}{{\tilde{\CY}}}
\nc{\TPhi}{{\tilde{\Phi}}}
\nc{\OPhi}{{\bar{\Phi}}}
\nc{\txi}{{\tilde{\xi}}}
\nc{\tp}{{\tilde{p}}}
\nc{\tq}{{\tilde{q}}}
\nc{\tzeta}{{\tilde{\zeta}}}
\nc{\tpi}{{\tilde{\pi}}}

\nc{\halpha}{{\hat{\alpha}}}
\nc{\HCA}{{\hat{\CA}}}
\nc{\HCB}{{\hat{\CB}}}
\nc{\HCC}{{\hat{\CC}}}
\nc{\HE}{{\widehat{\CE}}}
\nc{\HX}{{\hat{X}}}
\nc{\hxi}{{\hat{\xi}}}

\nc{\UH}{{\mathcal{H}}}

\nc{\TM}{{\widetilde{M}}}
\nc{\TCM}{{\widetilde{\CM}}}
\nc{\TU}{{\widetilde{U}}}
\nc{\TX}{{\widetilde{X}}}
\nc{\TY}{{\widetilde{Y}}}
\nc{\TYO}{{{\widetilde{Y}}^\circ}}
\nc{\barf}{{\bar{f}}}
\nc{\te}{{\tilde{e}}{}}
\nc{\tf}{{\tilde{f}}}
\nc{\tg}{{\tilde{g}}}
\nc{\ti}{{\tilde{\imath}}}
\nc{\tj}{{\tilde{\jmath}}}
\nc{\ty}{{\tilde{y}}}
\nc{\tphi}{{\tilde{\phi}}}

\nc{\urho}{{\underline{\rho}}}

\nc{\LRA}{\Leftrightarrow}
\nc{\RA}{\Rightarrow}
\nc{\lotimes}{\mathbin{\mathop{\otimes}\limits^{\mathbb{L}}}}
\nc{\CEnd}{\mathop{\mathcal{E}\mathit{nd}}\nolimits}
\nc{\CExt}{\mathop{\mathcal{E}\mathit{xt}}\nolimits}
\nc{\CHom}{\mathop{\mathcal{H}\mathit{om}}\nolimits}
\nc{\RH}{\mathop{{\mathsf{R}}\Gamma}\nolimits}
\nc{\RGamma}{\mathop{{\mathsf{R}}\Gamma}\nolimits}
\nc{\RHom}{\mathop{\mathsf{RHom}}\nolimits}
\nc{\RCHom}{\mathop{\mathsf{R}\mathcal{H}\mathit{om}}\nolimits}
\nc{\RG}{\mathop{\mathsf{R\Gamma}}\nolimits}
\nc{\Hom}{\mathop{\mathsf{Hom}}\nolimits}
\nc{\Ext}{\mathop{\mathsf{Ext}}\nolimits}
\nc{\End}{\mathop{\mathsf{End}}\nolimits}
\nc{\Tor}{\mathop{\mathsf{Tor}}\nolimits}
\nc{\Tordim}{\mathop{\mathsf{Tor}\text{\rm-}\mathsf{dim}}\nolimits}
\nc{\Hilb}{\mathop{\mathsf{Hilb}}\nolimits}
\nc{\Spec}{\mathop{\mathsf{Spec}}\nolimits}
\nc{\Pic}{\mathop{\mathsf{Pic}}\nolimits}
\renewcommand{\Im}{\mathop{\mathsf{Im}}\nolimits}
\nc{\Tr}{\mathop{\mathsf{Tr}}\nolimits}
\nc{\Cone}{\mathop{\mathsf{Cone}}\nolimits}
\nc{\Fiber}{\mathop{\mathsf{Fiber}}\nolimits}
\nc{\Ker}{\mathop{\mathsf{Ker}}\nolimits}
\nc{\Coker}{\mathop{\mathsf{Coker}}\nolimits}
\nc{\codim}{\mathop{\mathsf{codim}}\nolimits}
\nc{\sing}{{\mathsf{sing}}}
\nc{\supp}{\mathop{\mathsf{supp}}}
\nc{\vol}{\mathop{\mathsf{vol}}\nolimits}
\nc{\ch}{\mathop{\mathsf{ch}}\nolimits}
\nc{\perf}{{\mathsf{perf}}}
\nc{\rank}{\mathop{\mathsf{rank}}}
\nc{\Pf}{{\mathsf{Pf}}}
\nc{\Gr}{{\mathsf{Gr}}}
\nc{\OGr}{{\mathsf{OGr}}}
\nc{\Flag}{{\mathsf{Fl}}}
\nc{\Kosz}{{\mathsf{Kosz}}}
\nc{\IGr}{{\mathsf{IGr}}}
\nc{\LGr}{{\mathsf{LGr}}}
\nc{\SGr}{{\mathsf{SGr}}}
\nc{\GGT}{{\mathsf{G_2}}}
\nc{\GTGr}{{\mathsf{G_2Gr}}}
\nc{\GTF}{{\mathsf{G_2F}}}
\nc{\OF}{{\mathsf{OF}}}
\nc{\Fl}{{\mathsf{Fl}}}
\nc{\Bl}{{\mathsf{Bl}}}
\nc{\GL}{{\mathsf{GL}}}
\nc{\PGL}{{\mathsf{PGL}}}
\nc{\SL}{{\mathsf{SL}}}
\nc{\SP}{{\mathsf{Sp}}}
\nc{\Spin}{{\mathsf{Spin}}}
\nc{\Tot}{{\mathsf{Tot}}}
\nc{\ev}{{\mathsf{ev}}}
\nc{\od}{{\mathsf{odd}}}
\nc{\coev}{{\mathsf{coev}}}
\nc{\id}{{\mathsf{id}}}
\nc{\opp}{{\mathsf{opp}}}
\nc{\PS}{{{\PP^3}}}
\nc{\Qu}{{{Q^3}}}
\nc{\tdim}{\mathop{\Tor\dim}}
\nc{\ecart}{{\fbox{$\scriptstyle\mathsf{EC}$}}}
\nc{\ad}{{\mathop{\mathsf ad}}}
\nc{\sg}{{\mathop{\mathsf sg}}}
\nc{\hf}{{\mathop{\mathsf hf}}}
\nc{\gr}{{\mathop{\mathsf gr}}}
\nc{\qgr}{{\mathop{\mathsf qgr}}}
\nc{\Coh}{{\mathop{\mathsf Coh}}}
\nc{\Ab}{{\mathop{\mathcal{A}\mathit{b}}}}
\nc{\Ccoh}{{\mathop{\mathsf Ccoh}}}
\nc{\Qcoh}{{\mathop{\mathsf Qcoh}}}
\nc{\At}{\mathop{\mathsf{At}}}
\nc{\tra}{{\mathsf{T}}}
\nc{\fsl}{{\mathfrak{sl}}}
\nc{\fso}{{\mathfrak{so}}}
\nc{\fgl}{{\mathfrak{gl}}}

\nc{\AAV}{{\mathcal{AAV}}}

\nc{\Rep}{{\mathsf{Rep}}}

\nc{\Cubics}{{{\mathcal{S}}_3}}
\nc{\VFT}{{{\mathcal{S}}_{14}}}
\nc{\VFTE}{{{\mathcal{N}}_{\mathrm{reg,sm}}}}
\nc{\MX}{{\CM_X}}
\nc{\MY}{{\CM_Y}}
\nc{\MYE}{{\CM_{Y,\CE}}}
\nc{\Yd}{{Y_d}}
\nc{\Yfive}{{Y_5}}
\nc{\Xg}{{X_{2g-2}}}
\nc{\Xtt}{{X_{22}}}
\nc{\Xst}{{X_{16}}}
\nc{\Xtw}{{X_{12}}}
\nc{\Xe}{{X_{8}}}
\nc{\Xf}{{X_{4}}}

\nc{\git}{{/\!\!/\!{}_\chi}}

\nc{\HOH}{{\mathsf H\mathsf H}}
\nc{\HHE}{{\mathsf H\mathsf E}}

\theoremstyle{plain}

\newtheorem{theorem}{Theorem}[section]
\newtheorem{conjecture}[theorem]{Conjecture}

\newtheorem{lemma}[theorem]{Lemma}
\newtheorem{proposition}[theorem]{Proposition}
\newtheorem{corollary}[theorem]{Corollary}

\theoremstyle{definition}

\newtheorem{definition}[theorem]{Definition}

\newtheorem{example}[theorem]{Example}

\theoremstyle{remark}

\newtheorem{remark}[theorem]{Remark}

\newcommand{\epsi}{\epsilon}

\title{Calabi--Yau and fractional Calabi--Yau categories}
\author{Alexander Kuznetsov}
\address{{\sloppy
\parbox{0.9\textwidth}{
Steklov Mathematical Institute,
8 Gubkin str., Moscow 119991 Russia
\\[5pt]
The Poncelet Laboratory, Independent University of Moscow
\hfill\\[5pt]
Laboratory of Algebraic Geometry, National Research University Higher School of Economics
}\bigskip}}
\email{akuznet@mi.ras.ru}
\date{}
\thanks{I was partially supported by a subsidy granted to the HSE by the Government of the Russian Federation 
for the implementation of the Global Competitiveness Program and by RFBR 14-01-00416, 15-01-02164, 15-51-50045 and by the Simons foundation.}

\begin{document}

\begin{abstract}
We discuss Calabi--Yau and fractional Calabi--Yau semiorthogonal components of derived categories
of coherent sheaves on smooth projective varieties. The main result is a general construction of 
a fractional Calabi--Yau category from a rectangular Lefschetz decomposition and a spherical functor. 
We give many examples of application of this construction and discuss some general properties of Calabi--Yau categories.
\end{abstract}

\maketitle


\section{Introduction}

Projective varieties with trivial canonical class (Calabi--Yau varieties) form a very important class
of varieties in algebraic geometry. Their importance is emphasized by the special role they play
in Mirror Symmetry which associates with each Calabi--Yau variety $X$ its mirror partner $Y$, 
such that the Hodge numbers of $X$ and $Y$ are related by $h^{p,q}(Y) = h^{q,n-p}(X)$, where 
$n = \dim X = \dim Y$. However, this relation shows that by considering only usual Calabi--Yau varieties 
we are missing some mirror partners. Indeed, if $X$ is a rigid Calabi--Yau variety then $h^{n-1,1}(X) = 0$ 
and so one expects to have $h^{1,1}(Y) = 0$ for the mirror partner $Y$ of $X$. Thus $Y$ cannot be projective. 

It is expected, however, that Mirror Symmetry extends to rigid Calabi--Yau varieties, but their mirror
partners are non-commutative Calabi--Yau varieties. In other words, instead of an algebraic variety $Y$
one expects to associate with $X$ a certain triangulated category $\CT$ (thought of as the derived category
of coherent sheaves on a ``noncommutative variety $Y$''). 

To express the Calabi--Yau property of $Y$ in terms of $\CT$ it is natural to use the Serre functor $\SSS_\CT$.
The Serre functor is one of the most important invariants of a triangulated category (see Section~\ref{subsection-serre}),
which for derived categories of coherent sheaves on smooth projective varieties is the composition of the twist by the canonical class 
and the shift by the dimension. Thus derived categories of Calabi--Yau varieties are characterized by the fact that their Serre functor
is just a shift. This motivates the following

\begin{definition}
A triangulated category $\CT$ is an {\sf $n$-Calabi--Yau category} if it has a Serre functor $\SSS_\CT$
and, moreover, $\SSS_\CT \cong [n]$ for some $n \in \ZZ$.
The integer $n$ is called the {\sf CY-dimension of $\CT$}.
\end{definition}

It is also natural to consider the following weakening of the Calabi--Yau property.

\begin{definition}
A triangulated category $\CT$ is a {\sf fractional Calabi--Yau category} if it has a Serre functor $\SSS_\CT$
and there are integers $p$ and $q \ne 0$ such that $\SSS_\CT^q \cong [p]$.
\end{definition}

The goal of this paper is to show that there are many examples of Fano varieties which have a semiorthogonal decomposition
with one of the components being a fractional Calabi--Yau category. The presence of a Calabi--Yau component usually has a strong influence
on the geometrical properties of the Fano variety, which acquires some properties specific to Calabi--Yau varieties
(this was discussed from the Hodge-theoretic point of view in~\cite{iliev2011fano}). For example, if a variety $X$
has a semiorthogonal component which is 2-Calabi--Yau category then any moduli space of coherent sheaves on $X$
carries a closed 2-form, and some of them provide interesting examples of hyper-k\"ahler varieties. 
This makes it interesting to find some general construction of Calabi--Yau categories of geometric origin.

The main result of this paper is such a construction. We start with a smooth projective variety $M$
with a rectangular Lefschetz decomposition (see Section~\ref{subsection-sod} for a definition and
Section~\ref{subsection-lefschetz} for examples of such varieties, the simplest example to have in mind
is the projective space $\PP^n$, or the Grassmannian $\Gr(k,n)$ with coprime $k$ and $n$).
Further, consider a spherical functor $\Phi:\BD(X) \to \BD(M)$ between the bounded derived categories
of coherent sheaves of another smooth projective variety $X$ and $M$ (see Section~\ref{subsection-spherical}
for a definition and Section~\ref{sec-setup} for some examples, again the simplest example is the derived
pushforward for a divisorial embedding $X \hookrightarrow M$).
Assuming some compatibility between the Lefschetz
decomposition of $\BD(M)$ and the functor $\Phi$ we prove that $\BD(X)$ has a semiorthogonal decomposition, 
such that an appropriate power of the Serre functor of one of the components of this decomposition is
isomorphic to a shift. The construction is explained in detail in Section~\ref{section-construction}
after a preparatory Section~\ref{section-preliminaries}.

In Section~\ref{section-examples} we list some known varieties with a rectangular
Lefschetz decomposition and some Calabi--Yau categories arising from these. We pay special
attention to K3 and 3-Calabi--Yau categories coming from these examples.

Finally, in Section~\ref{section-properties} we discuss some general properties of Calabi--Yau categories.
We show that a connected Calabi--Yau category is indecomposable, and prove an inequality between
the CY-dimension of a Calabi--Yau component of the derived category of a smooth projective variety
and the dimension of the variety itself. We also discuss some interesting questions and conjectures 
related to Calabi--Yau categories.

I would like to thank Alex Perry for suggestion to consider Example~\ref{example-root} of a spherical functor
and for many valuable comments on the first draft of the paper.

\section{Preliminaries}\label{section-preliminaries}

\subsection{Notations and conventions}

All varieties considered in this paper are assumed to be smooth and projective over a field $\kk$.
In the examples related to Grassmannians the field is assumed to be of zero (or sufficiently big positive)
characteristic. For a variety $X$ we denote by $\BD(X)$ the bounded derived category of coherent sheaves on $X$.
All the pushforward, pullback, and tensor product functors are derived. All functors between triangulated
categories are assumed to be triangulated. For a functor $\Phi:\CT_1 \to \CT_2$ between triangulated categories 
$\CT_1$ and $\CT_2$ we denote by $\Phi^*$ its left adjoint and by $\Phi^!$ its right adjoint (if they exist).
We denote the units and the counits of the adjunctions by
\begin{equation*}
\eta_{\Phi,\Phi^*}:\id \to \Phi\circ\Phi^*
\qquad\text{and}\qquad
\epsi_{\Phi^*,\Phi}:\Phi^*\circ\Phi \to \id,
\end{equation*}
and if there is no risk of confusion we omit the lower indices. Recall that the compositions
\begin{equation*}
\Phi \xrightarrow{\ \eta_{\Phi,\Phi^*} \circ \Phi \ } \Phi \circ \Phi^* \circ \Phi \xrightarrow{\ \Phi \circ \epsi_{\Phi^*,\Phi}\ } \Phi
\qquad\text{and}\qquad
\Phi^* \xrightarrow{\ \Phi^* \circ \eta_{\Phi,\Phi^*} \ } \Phi^* \circ \Phi \circ \Phi^* \xrightarrow{\ \epsi_{\Phi^*,\Phi} \circ \Phi^*\ } \Phi^*
\end{equation*}
are identity morphisms (in fact, this is one of the equivalent definitions of adjunction).

Given an object $\CE \in \BD(X\times Y)$ we can consider a functor
\begin{equation*}
\BD(X) \to \BD(Y),
\qquad
F \mapsto p_{Y*}(\CE \otimes p_X^*(F)),
\end{equation*}
where $p_X$ and $p_Y$ are the projections of $X\times Y$ to $X$ and $Y$ respectively. It is called the
{\sf Fourier--Mukai functor with kernel $\CE$}. A morphism of kernels induces a morphism
of the corresponding Fourier--Mukai functors. Furthermore there is an operation of {\sf convolution} of kernels,
which corresponds to composition of functors. Finally, any Fourier--Mukai functor
has both adjoints which are also Fourier--Mukai functors, and moreover, the unit and the counit 
of the adjunctions are induced by morphisms of kernels \cite{anno2012adjunctions}.

In what follows, to unburden notation we will identify Fourier--Mukai functors with their kernels,
and we will consider only those morphisms of functors which are induced by morphims of kernels.
In particular, by a distinguished triangle of (Fourier--Mukai) functors we understand a distinguished 
triangle of kernels.
Furthermore, any object of $\BD(X)$ will be identified with the derived tensor product functor $\CF \otimes -$,
i.e.\ with Fourier--Mukai functor whose kernel is the pushforward of $\CF$ to $X \times X$ under the diagonal embedding.
Thus given a line bundle $\CL_X$ on $X$ the same notation will be used for the tensor product $\CL_X \otimes - $ functor.
Similarly, given an automorphism $\tau$ of $X$ we will write $\tau$ also for the autoequivalence of $\BD(X)$ it induces.

\subsection{Semiorthogonal decompositions and mutation functors}\label{subsection-sod}

For a review of semiorthogonal decompositions and their uses one can look into~\cite{kuznetsov2014semiorthogonal}.

\begin{definition}
A {\sf semiorthogonal decomposition} of a triangulated category $\CT$ is a collection $\CA_1,\dots,\CA_m$
of full triangulated subcategories in $\CT$ such that
\begin{itemize}
\item 
for all $i > j$ we have $\Hom(\CA_i,\CA_j) = 0$;
\item 
for any object $T \in \CT$ there is a filtration, i.e., a chain of morphisms 
\begin{equation*}
0 = T_m \to T_{m-1} \to \dots \to T_1 \to T_0 = T
\end{equation*}
such that $\Cone(T_i \to T_{i-1}) \in \CA_i$.
\end{itemize}
A semiorthogonal decomposition is denoted by $\CT = \langle \CA_1, \dots, \CA_m \rangle$.
\end{definition}

The filtration
in the second part of the definition is canonical and functorial. Moreover, if $\CT = \BD(X)$
is the derived category of a smooth projective variety the fitration of every object is induced by a filtration of the structure sheaf of the diagonal
in the following sense.

\begin{lemma}[\cite{kuznetsov2011base}]\label{lemma-kernels-projectors}
If\/ $\BD(X) = \langle \CA_1,\dots,\CA_m \rangle$ is a semiorthogonal decomposition then there is a chain of morphisms in $\BD(X\times X)$
\begin{equation*}
0 = \Delta_m \to \Delta_{m-1} \to \dots \to \Delta_1 \to \Delta_0 = \Delta_*\CO_X
\end{equation*}
such that for any $T \in \BD(X)$ one has $T_i = \Delta_i(T)$, where each $\Delta_i$ is considered as a Fourier--Mukai functor $\Delta_i:\BD(X) \to \BD(X)$.
In particular, the projection functors onto components of a semiorthogonal decomposition are Fourier--Mukai functors.
\end{lemma}

If $X$ is a smooth projective variety and $\BD(X) = \langle \CA_1, \dots, \CA_m \rangle$ is a semiorthogonal decomposition 
then each component $\CA_i \subset \BD(X)$ is {\sf admissible} (see~\cite{bondal-kapranov}). This means that its embedding functor $\alpha_i:\CA_i \to \BD(X)$
has both left and right adjoints $\alpha_i^*:\BD(X) \to \CA_i$ and $\alpha_i^!:\BD(X) \to \CA_i$ (note that by full faithfulness
of $\alpha_i$ it follows that $\alpha_i^*\alpha_i = \alpha_i^!\alpha_i = \id_{\CA_i}$).

Vice versa, any semiorthogonal collection $\CA_1,\dots,\CA_m$ of admissible triangulated subcategories in a triangulated category $\CT$
extends to a semiorthogonal decomposition $\CT = \langle \CA, \CA_1, \dots, \CA_m \rangle$ with an additional component
$\CA$ defined as the orthogonal
\begin{equation*}
\CA = \langle \CA_1, \dots, \CA_m \rangle^\perp := \{ T \in \CT \mid \Hom(\CA_i,T) = 0 \}.
\end{equation*}
Actually, instead of adding the component $\CA$ on the left of the collection, we could extend the collection
to a semiorthogonal decomposition by choosing any $1 \le i \le m$ and inserting an appropriate intersection
of orthogonals between $\CA_i$ and $\CA_{i+1}$.
In particular, if $\CB \subset \CT$ is an admissible subcategory then it extends in two ways
to a semiorthogonal decomposition
\begin{equation*}
\CT = \langle \CB^\perp , \CB \rangle,
\qquad
\CT = \langle \CB, {}^\perp\CB \rangle,
\end{equation*}
The additional components $\CB^\perp$ and ${}^\perp\CB$ of $\CT$ are abstractly equivalent but embedded into $\CT$ differently.
An equivalence between these subcategories is given by {\sf mutation functors}.

The {\sf left mutation functor} through $\CB$ is denoted $\LL_\CB$ and is defined by the canonical functorial distinguished triangle
\begin{equation}\label{def-left-mut}
\beta\beta^! \xrightarrow{\ \epsi\ } \id \to \LL_\CB,
\end{equation}
where $\beta:\CB \to \CT$ is the embedding functor.
Analogously, the {\sf right mutation functor} through $\CB$ is denoted~$\RR_\CB$ and is defined by the canonical functorial distinguished triangle
\begin{equation}\label{def-right-mut}
\RR_\CB \to \id \xrightarrow{\ \eta\ } \beta\beta^*.
\end{equation}
The following two results about mutations are straightforward, but quite useful.

\begin{lemma}\label{lemma-mutations-basics}
$(i)$ If $\CB = \langle \CB_1, \dots, \CB_k \rangle$ is a semiorthogonal decomposition of an admissible subcategory $\CB \subset \CT$ then
\begin{equation*}
\LL_{\CB} = \LL_{\CB_1} \circ \dots \circ \LL_{\CB_k}
\qquad\text{and}\qquad
\RR_{\CB} = \RR_{\CB_k} \circ \dots \circ \RR_{\CB_1}.
\end{equation*}

$(ii)$ If $\xi:\CT \to \CT$ is an autoequivalence then
\begin{equation*}
\xi \circ \LL_\CB \circ \xi^{-1} = \LL_{\xi(\CB)}
\qquad\text{and}\qquad
\xi \circ \RR_\CB \circ \xi^{-1} = \RR_{\xi(\CB)}.
\end{equation*}
\end{lemma}

Assume $M$ is a smooth projective variety and $\CL_M$ is a line bundle on $M$. 
A {\sf Lefschetz decomposition} of $\BD(M)$
is a semiorthogonal decomposition in which each component is embedded into the $\CL_M$ twist of the previous component.
The formal definition is:

\begin{definition}\label{definition:ld}
A {\sf Lefschetz decomposition} of $\BD(M)$
is a semiorthogonal decomposition of the form
\begin{equation*}
\BD(M) = \langle \CB_0 ,\CB_1\otimes\CL_M, \dots, \CB_{m-1} \otimes \CL_M^{m-1} \rangle,
\qquad\text{where $\CB_0 \supset \CB_1 \supset \dots \supset \CB_{m-1}$.}
\end{equation*}
A Lefschetz decomposition is {\sf rectangular} if $\CB_0 = \CB_1 = \dots = \CB_{m-1}$.
\end{definition}

\subsection{Serre functor}\label{subsection-serre}

One of the main characteristics of a triangulated category is its Serre functor.

\begin{definition}[\cite{bondal-kapranov}]
Let $\CT$ be a triangulated category. A {\sf Serre functor} in $\CT$ is an autoequivalence $\SSS_\CT:\CT \to \CT$
with a bifunctorial isomorphism
\begin{equation*}
\Hom(F,G)^\vee \cong \Hom(G,\SSS_\CT(F))
\end{equation*}
for all $F,G \in \CT$.
\end{definition}

If a Serre functor exists then it is unique up to a canonical isomorphism.
If $\CT = \BD(X)$ is the bounded derived category of a smooth projective variety $X$ then
\begin{equation*}
\SSS_X(F) := F \otimes \omega_X[\dim X]
\end{equation*}
is a Serre functor for $\BD(X)$. 

The following properties of Serre functors are quite useful.

\begin{lemma}\label{lemma-serre-basics}
$(i)$ Let $\CT_1$ and $\CT_2$ be triangulated categories with Serre functors $\SSS_{\CT_1}$ and $\SSS_{\CT_2}$ respectively.
If $\Phi:\CT_1 \to \CT_2$ is a functor then its left adjoint $\Phi^*$ exists if and only if its right adjoint $\Phi^!$ exists and
\begin{equation*}
\Phi^! \circ \SSS_{\CT_2} = \SSS_{\CT_1} \circ \Phi^*.
\end{equation*}

$(ii)$ The Serre functor of a triangulated category $\CT$ commutes with all its autoequivalences.
\end{lemma}

Another useful feature is a relation of the Serre functor of a triangulated category with Serre functors of components of its semiorthogonal decomposition.

\begin{lemma}\label{lemma-serre-mutation}
Let $\CT = \langle \CA, \CB \rangle$ be a semiorthogonal decomposition with admissible $\CA$ and $\CB$, 
and assume that a Serre functor of~$\CT$ exists. Then Serre functors of $\CA$ and $\CB$ exist and
\begin{equation*}
\SSS_{\CB} = \RR_\CA \circ \SSS_\CT,
\qquad\text{and}\qquad
\SSS_{\CA}^{-1} = \LL_{\CB} \circ \SSS_\CT^{-1}.
\end{equation*}
\end{lemma}

The following compatibility with rectangular Lefschetz decompositions will be useful later.

\begin{lemma}\label{lemma:serre-lefschetz}
Let $M$ be a smooth projective variety and $\BD(M) = \langle \CB, \CB \otimes \CL_M, \dots, \CB \otimes \CL_M^{m-1} \rangle$ a rectangular Lefschetz decomposition. 
Then for each $i \in \ZZ$ one has $\SSS_M(\CB \otimes \CL_M^i) = \CB \otimes \CL_M^{i-m}$.
\end{lemma}
\begin{proof}
First, tensoring the decomposition by $\CL_M^{i-m+1}$ we deduce that
\begin{equation*}
\CB \otimes \CL_M^{i} = {}^\perp\langle \CB \otimes \CL_M^{i-m+1}, \dots, \CB \otimes \CL_M^{i-1}\rangle.
\end{equation*}
From the definition of a Serre functor it then follows that
\begin{equation*}
\SSS_M(\CB \otimes \CL_M^{i}) = \langle \CB \otimes \CL_M^{i-m+1}, \dots, \CB \otimes \CL_M^{i-1}\rangle^\perp.
\end{equation*}
Comparing this with the initial decomposition tensored by $\CL_M^{i-m}$, we deduce the required equality.
\end{proof}

\subsection{Hochschild homology and cohomology}\label{subsection-hochschild}

Hochschild homology $\HOH_\bullet(\CT)$ and cohomology $\HOH^\bullet(\CT)$ are important invariants of triangulated categories. 
One of the ways to define them is by choosing an appropriate DG-enhancement for $\CT$ and using Hochschild homology and cohomology 
of DG-categories (see~\cite{keller2006differential}). However, for derived categories of smooth projective varieties and their 
semiorthogonal components one can use Fourier--Mukai kernels as a replacement for an enhancement. For details we refer 
to~\cite{kuznetsov2009hochschild} and here just sketch the main results.

\begin{lemma}[\cite{kuznetsov2009hochschild}]
Let $\CA \subset \BD(X)$ be an admissible subcategory and $P \in \BD(X\times X)$ the Fourier--Mukai kernel of the projection
functor onto $\CA$. Then
\begin{equation*}
\HOH^\bullet(\CA) = \Ext^\bullet(P,P),
\qquad
\HOH_\bullet(\CA) = \Ext^\bullet(P,P \circ \SSS_X).
\end{equation*}
\end{lemma}

For $\CA = \BD(X)$ the Hochschild homology and cohomology are related to classical invariants of $X$
via the Hochschild--Kostant--Rosenberg isomorphism (HKR for short):
\begin{equation}
\HOH^n(\BD(X)) = \bigoplus_{p+q = n} H^q(X,\Lambda^pT_X),
\qquad
\HOH_n(\BD(X)) = \bigoplus_{q-p = n} H^q(X,\Omega^p_X).
\end{equation} 

The Hochschild cohomology of any category has a structure of a graded algebra (and moreover, of a Gerstenhaber algebra),
and the Hochschild homology is a right module over it. Hochschild homology has a nice additivity property.

\begin{lemma}[\cite{kuznetsov2009hochschild}]\label{lemma-hoh-additivity}
If $\CA = \langle \CA_1, \dots, \CA_m \rangle$ is a semiorthogonal decomposition then
\begin{equation*}
\HOH_\bullet(\CA) = \HOH_\bullet(\CA_1) \oplus \dots \oplus \HOH_\bullet(\CA_m).
\end{equation*}
\end{lemma}

Hochschild cohomology is additive only for completely orthogonal decompositions.
On the other hand, it has a nice nonvanishing property.

\begin{lemma}\label{lemma-choh-additivity}
If $\CT = \langle \CA, \CB \rangle$ is a completely orthogonal decomposition, i.e. $\Hom(\CA,\CB) = \Hom(\CB,\CA) = 0$, then
\begin{equation*}
\HOH^\bullet(\CT) = \HOH^\bullet(\CA) \oplus \HOH^\bullet(\CB).
\end{equation*}
If $\CA \ne 0$ then $\HOH^0(\CA) \ne 0$.
\end{lemma}
\begin{proof}
The first follows from \cite[Thm.~7.7]{kuznetsov2009hochschild}. For the second note that for nonzero $\CA$ the corresponding 
projection kernel $P$ is nonzero, and hence has a nonzero endomorphism (the identity).
\end{proof}

\subsection{Spherical functors}\label{subsection-spherical}

Spherical functors were introduced in~\cite{anno2007spherical}, see also~\cite{anno2013spherical}
for a more recent development. 
The following is equivalent to the classical definition.

\begin{definition}[cf.~\cite{anno2007spherical}]\label{defsph}
A Fourier--Mukai functor $\Phi:\BD(X) \to \BD(Y)$ is {\sf spherical} if
\renewcommand{\theenumi}{\roman{enumi}}%
\begin{enumerate}
\item
the map $\Phi^* \oplus \Phi^! \xrightarrow{\ \eta_{\Phi^!,\Phi}\circ\Phi^* + \Phi^!\circ\eta_{\Phi,\Phi^*}\ } \Phi^!\circ\Phi\circ\Phi^*$
is an isomorphism, and
\item
the map $\Phi^*\circ\Phi\circ\Phi^! \xrightarrow{\ \Phi^*\circ\epsi_{\Phi,\Phi^!} + \epsi_{\Phi^*,\Phi}\circ\Phi^!\ } \Phi^* \oplus \Phi^!$
is an isomorphism.
\end{enumerate}
\end{definition}

\begin{proposition}\label{proposition-spherical-twists}
If the conditions of Definition~\ref{defsph} are satisfied then the functors $T_X$ and $T'_X$ as well as the functors $T_Y$ and $T'_Y$
defined by the following distinguished triangles
\begin{align}
T_Y \xrightarrow{\qquad} &\id \xrightarrow{\ \eta_{\Phi,\Phi^*}\ } \Phi\circ\Phi^*,\label{defty}\\
\Phi^*\circ\Phi \xrightarrow{\ \epsi_{\Phi^*,\Phi}\ } &\id \xrightarrow{\qquad} T_X,\label{deftx}\\
\Phi\circ\Phi^! \xrightarrow{\ \epsi_{\Phi,\Phi^!}\ } &\id \xrightarrow{\qquad} T'_Y,\label{deftpy}\\
T'_X \xrightarrow{\qquad} &\id \xrightarrow{\ \eta_{\Phi^!,\Phi}\ } \Phi^!\circ\Phi.\label{deftpx}
\end{align}
are mutually inverse autoequivalences of $\BD(X)$ and $\BD(Y)$.
\end{proposition}

The idea behind the proof is very simple --- assuming equality $abc = a + c$ one can deduce from it $(1-ab)(1-cb) = 1$ by multiplying 
the equality with $b$. The argument below is a categorical version of this taking care of all the subtleties.

\begin{proof}
Denote the connecting morphism $\Phi^!\circ\Phi \to T'_X[1]$ in~\eqref{deftpx} by $\delta$.
Composing a rotation of~\eqref{deftpx} with~$\Phi^*$ on the right we get a distinguished triangle
\begin{equation*}
\Phi^* \xrightarrow{\ \eta_{\Phi^!,\Phi}\circ\Phi^*\ } \Phi^!\circ\Phi\circ\Phi^* \xrightarrow{\quad\delta\circ\Phi^*\quad} T'_X[1] \circ \Phi^*.
\end{equation*}
Using Definition~\ref{defsph}(i) we conclude that the map $(\delta\circ\Phi^*)\circ (\Phi^!\circ\eta_{\Phi,\Phi^*}) : \Phi^! \to \Phi^!\circ\Phi\circ\Phi^* \to T'_X[1] \circ \Phi^*$
is an isomorphism. We multiply this with $\Phi$ on the right and check that the composition of the resulting morphism with $T'_X[1]\circ\epsi_{\Phi^*,\Phi}:T'_X[1]\circ\Phi^*\circ\Phi \to T'_X[1]$ 
coincides with $\delta$. This follows from the diagram
\begin{equation*}
\xymatrix{
\Phi^!\circ\Phi \ar[rrr]^-{\Phi^!\circ\eta_{\Phi,\Phi^*}\circ\Phi} \ar@{-->}[drrr] &&&
\Phi^!\circ\Phi\circ\Phi^*\circ\Phi \ar[rrr]^-{\delta\circ\Phi^*\circ\Phi} \ar[d]^{\Phi^!\circ\Phi\circ\epsi_{\Phi^*,\Phi}} &&&
T'_X[1]\circ\Phi^*\circ\Phi \ar[d]^{T'_X[1]\circ\epsi_{\Phi^*,\Phi}}
\\
&&&
\Phi^!\circ\Phi \ar[rrr]^-{\delta} &&&
T'_X[1]
}
\end{equation*}
Indeed, the square commutes since the vertical and the horizontal arrows in it act on different variables,
and the diagonal dashed arrow is the identity by the standard characterization of adjunction (composed with $\Phi^!$ on the left). 
This means that in the diagram
\begin{equation*}
\xymatrix{
\id \ar[rrr]^-{\eta_{\Phi^!,\Phi}} \ar@{..>}[d] &&&
\Phi^!\circ\Phi \ar[rrr]^-{\delta} \ar[d]_{(\delta\circ\Phi^*\circ\Phi)\circ (\Phi^!\circ\eta_{\Phi,\Phi^*}\circ\Phi)} &&&
T'_X[1] \ar@{=}[d]
\\
T'_X\circ T_X \ar[rrr] &&&
T'_X[1]\circ\Phi^*\circ\Phi \ar[rrr]^-{T'_X[1]\circ\epsi_{\Phi^*,\Phi}} &&&
T'_X[1]
}
\end{equation*}
where the top line is the distinguished triangle~\eqref{deftpx} and the bottom line is the distinguished triangle~\eqref{deftx} composed with $T'_X[1]$ on the left,
the right square is commutative. Since the vertical arrows are isomorphisms, it follows that there is a dotted vertical arrow on the left,
which is also an isomorphism. 

Thus $T'_X \circ T_X \cong \id$.
Analogously one proves that the other compositions are isomorphic to the identity. For $T'_Y\circ T_Y$ part $(i)$ of Definition~\ref{defsph} is used,
while for $T_X\circ T'_X$ and for $T_Y\circ T'_Y$ part~$(ii)$ is used.
\end{proof}

\begin{remark}
It may well be that it is enough to assume only one of the conditions of Definition~\ref{defsph}.
Indeed, assuming for example part $(i)$ we can prove that the compositions $T'_X\circ T_X$ and $T'_Y\circ T_Y$ are isomorphic to identity.
On the other hand, it is easy to see that $T'_X$ and $T'_Y$ are right adjoint to $T_X$ and $T_Y$ respectively. So, it follows
that $T_X$ and $T_Y$ are fully faithful endofunctors. It is very tempting to conjecture that any such endofunctor 
of the derived category of a smooth projective variety is an autoequivalence --- then it would follow that $T'_X$ and
$T'_Y$ are quasiinverse of $T_X$ and $T_Y$ and so are also autoequivalences. Up to now it is not clear
how this conjecture can be proved. However, it can be easily deduced from the following
\end{remark}

\begin{conjecture}[noetherian property]
Any decreasing chain $\BD(X) = \CA_0 \supset \CA_1 \supset \CA_2 \supset \dots$ of admissible subcategories stabilizes,
i.e.\ for sufficiently large $n$ one has $\CA_i = \CA_{i+1}$ for all $i \ge n$.
\end{conjecture}

We will give examples of spherical functors in the next section (Examples~\ref{example-divisor}, \ref{example-covering}, and~\ref{example-root}).
For completeness we show that Definition~\ref{defsph} is equivalent to the standard one.

\begin{proposition}\label{proposition-new-old-definitions}
Definition~{\rm\ref{defsph}}\/ is equivalent to the original definition of a spherical functor in~\cite{anno2007spherical}.
\end{proposition}
\begin{proof}
Recall that the original definition amounted to assuming $T'_X$ is an autoequivalence and the map 
$(\delta\circ\Phi^*)\circ (\Phi^!\circ\eta_{\Phi,\Phi^*}) : \Phi^! \to \Phi^!\circ\Phi\circ\Phi^* \to T'_X[1] \circ \Phi^*$
is an isomorphism. As we already proved both these properties in Proposition~\ref{proposition-spherical-twists}, 
it follows that Definition~\ref{defsph} implies the one in~\cite{anno2007spherical}.

For the converse we compose the triangle~\eqref{deftpx} with $\Phi^*$ on the right and consider the commutative diagram
with the top line being the trivial triangle
\begin{equation*}
\xymatrix{
\Phi^* \ar[rrr] \ar@{=}[d] &&&
\Phi^* \oplus \Phi^! \ar[rrr] \ar[d]^{\eta_{\Phi^!,\Phi}\circ\Phi^* + \Phi^!\circ\eta_{\Phi,\Phi^*}} &&&
\Phi^! \ar[d]^{(\delta\circ\Phi^*)\circ (\Phi^!\circ\eta_{\Phi,\Phi^*})} \\
\Phi^* \ar[rrr]^-{\eta_{\Phi^!,\Phi}\circ\Phi^*} &&&
\Phi^!\circ\Phi \circ\Phi^* \ar[rrr]^-{\delta\circ\Phi^*} &&&
T'_X[1]\circ\Phi^*
}
\end{equation*}
It is easy to see that this is a morphism of triangles. Moreover, the left and the right vertical arrows are isomorphisms, 
hence so is the middle arrow. Further, by Proposition~1 of~\cite{anno2007spherical}
we know that $T_X$ is quasiinverse to $T'_X$, so it follows that $\Phi^* \cong T_X[-1] \circ \Phi^!$. Then the same argument as above proves 
that $\Phi^*\circ\Phi\circ\Phi^! \cong \Phi^* \oplus \Phi^!$.
\end{proof}

One of the advantages of Definition~\ref{defsph} in comparison with the original definition is that 
it uses neither the triangulated structures nor enhancements of $\BD(X)$ and $\BD(Y)$, and can be used 
for arbitrary functors between additive categories.
Further on we will also use the following standard property

\begin{corollary}\label{corollary-t-phi}
If $\Phi$ is a spherical functor and $T_X$ and $T_Y$ are the autoequivalences of $\BD(X)$ and $\BD(Y)$ defined by~\eqref{deftx} and~\eqref{defty} respectively, then
there are canonical isomorphisms
\begin{equation*}
\Phi \circ T_X  \cong T_Y \circ \Phi \circ [2]
\qquad\text{and}\qquad
T_X \circ \Phi^* \cong \Phi^* \circ T_Y \circ [2].
\end{equation*}
\end{corollary}
\begin{proof}
We already showed in the proof of Proposition~\ref{proposition-new-old-definitions} that $\Phi^*[1] \cong T_X \circ \Phi^!$.
An analogous argument, using~\eqref{defty} shows that $\Phi^*[-1] \cong \Phi^! \circ T_Y$.
Combining these two isomorphisms we conclude that 
\begin{equation*}
\Phi^* \circ T_Y^{-1} [-1] \cong \Phi^! \cong T_X^{-1} \circ \Phi^*[1]. 
\end{equation*}
Multiplying with~$T_X$ on the left
and with $T_Y[1]$ on the right we deduce the second isomorphism. Furthermore, passing to the right adjoint functors (and shifting by 1) we deduce the first isomorphism.
\end{proof}

\section{A construction of fractional Calabi--Yau categories}\label{section-construction}

\subsection{The setup}\label{sec-setup}

Assume we are given a smooth projective variety (or a stack) $M$ with a rectangular Lefschetz decomposition of length $m$
with respect to a line bundle $\CL_M$ (see Definition~\ref{definition:ld}). 
Assume also given another smooth projective variety (or a stack) $X$ and a spherical functor
$\Phi:\BD(X) \to \BD(M)$ which is compatible with the Lefschetz decomposition
in a certain way. Before explaining the compatibility conditions, let us first discuss
a number of model situations. In all these examples, in fact, the functor $\Phi$ is 
the (derived) pushforward for a morphism $f:X \to M$.

\begin{example}\label{example-divisor}
The map $f:X \to M$ is a divisorial embedding with the image $f(X)$ being a divisor in the linear system $\CL_M^d$ for some $1 \le d \le m$.
\end{example}

\begin{example}\label{example-covering}
The map $f:X \to M$ is a double covering branched in a divisor in the linear system~$\CL_M^{2d}$, again for some $1 \le d \le m$.
\end{example}

The third example is very similar to the second, but has some special features.

\begin{example}\label{example-root}
Let $\tf:\TX \to \TM$ be a double covering branched in a divisor in the linear system $\CL_\TM^{2d}$ for some $1 \le d \le m$.
This morphism is $\mu_2$-equivariant, where the group $\mu_2 = \{ \pm 1 \}$ acts on $\TX$ via the covering involution,
and on $\TM$ trivially. Let $X = [\TX/\mu_2]$, $M = [\TM/\mu_2]$ be the quotient stacks (thus $X$ is $\TM$ with 
the $\mu_2$-stacky structure along the branch divisor of $\tf$, while $M$ is $\TM$ with the $\mu_2$-stacky 
structure everywhere). The map $\tf$ descends to a map $X \to M$ which we denote by $f$.
\end{example}

In the next Proposition we check that in all these cases the functor $\Phi = f_* : \BD(X) \to \BD(M)$
is spherical, compute the corresponding spherical twists $T_M$ and $T_X$, and check some of their properties.
In all cases we denote $\CL_X := f^*\CL_M$, the pullback of the line bundle $\CL_M$ to $X$. Recall that according
to our conventions we also denote by $\CL_M$ and $\CL_X$ the autoequivalences of $\BD(M)$ and $\BD(X)$ defined as tensor
products with $\CL_M$ and $\CL_X$ respectively.

\begin{proposition}\label{proposition-examples}
Let $f: X \to M$ be a map from either of Examples~\ref{example-divisor}, \ref{example-covering}, or~\ref{example-root}.
Then the functor $\Phi = f_*:\BD(X) \to \BD(M)$ is spherical. 
Moreover, the spherical twist $T_X$ commutes with $\CL_X$, 
and an appropriate power of the functor
$\rho = T_X \circ \CL_X^d \cong \CL_X^d \circ T_X$ 
is a shift.
Finally, if $\omega_M = \CL_M^{-m}$ then an appropriate power of the functor $\sigma = \SSS_X \circ T_X \circ \CL_X^m$ is also a shift.
\end{proposition}
\begin{proof}
First assume that $f:X \to M$ is as in Example~\ref{example-divisor}. 
The relative canonical class is $\omega_{X/M} = \CL_X^d$, the relative dimension is $-1$, 
hence $f^!(F) \cong f^*(F)\otimes \CL_X^d[-1]$. Therefore
\begin{multline*}
f^!(f_*(f^*(F))) \cong 
f^!(F \otimes f_*\CO_X) \cong 
f^!(F \otimes (\CL_M^{-d} \xrightarrow{\ \varphi\ } \CO_M)) \cong \\ \cong
f^*(F) \otimes (\CL_X^{-d}[1] \oplus \CO_X) \otimes \CL_X^d[-1] \cong
f^*(F) \oplus f^!(F).
\end{multline*}
Here the first isomorphism is the projection formula, 
the second is the Koszul resolution for $f_*\CO_X$ (with~$\varphi$ being the equation of $X$ in $M$),
the third is the definition of $f^!$ combined with the fact that $\varphi_{|X} = 0$,
and the fourth is the definition of $f^!$ again. Computing analogously the composition $f^*\circ f_* \circ f^!$
we see that Definition~\ref{defsph} holds, so $f_*$ is a spherical functor. Finally, the standard distinguished triangles
\begin{equation*}
F \otimes \CL_M^{-d} \to F \to f_*f^*(F)
\qquad\text{and}\qquad
F \otimes \CL_X^{-d}[1] \to f^*f_*(F) \to F
\end{equation*}
show that in this case the spherical twists are 
\begin{equation}\label{equation-divisor-twists}
T_M = \CL_M^{-d}
\qquad\text{and}\qquad
T_X = \CL_X^{-d}[2].
\end{equation} 
Clearly $T_X$ commutes with~$\CL_X$ and 
\begin{equation}\label{equation-divisor-rs}
\rho = \CL_X^{-d}[2] \circ \CL_X^d = [2],
\qquad 
\sigma = \CL_X^{d} \circ f^*\omega_M[\dim M-1] \circ \CL_X^{-d}[2] \circ \CL_X^m = f^*(\omega_M\otimes\CL_M^m)[\dim M + 1],
\end{equation}
so with our assumptions both these functors are shifts.

\bigskip

Now assume that $f:X \to M$ is as in Example~\ref{example-covering}. 
Then the relative canonical class is again $\omega_{X/M} = \CL_X^d$, but the relative dimension is $0$, 
hence $f^!(F) \cong f^*(F)\otimes \CL_X^d$. Therefore
\begin{multline*}
f^!(f_*(f^*(F))) \cong 
f^!(F \otimes f_*\CO_X) \cong 
f^!(F \otimes (\CL_M^{-d} \oplus \CO_M)) \cong 
f^*(F) \otimes (\CL_X^{-d} \oplus \CO_X) \otimes \CL_X^d \cong
f^*(F) \oplus f^!(F).
\end{multline*}
Here again, the first isomorphism is the projection formula, 
the second is the definition of the double covering,
the third and the fourth is the definition of $f^!$. Computing analogously the composition $f^*\circ f_* \circ f^!$
we see that Definition~\ref{defsph} holds, so $f_*$ is a spherical functor. Finally, the standard distinguished triangles
\begin{equation*}
F \to f_*f^*(F) \to F \otimes \CL_M^{-d} 
\qquad\text{and}\qquad
\tau^*F \otimes \CL_X^{-d} \to f^*f_*(F) \to F,
\end{equation*}
where $\tau$ is the involution of the covering, show that in this case the spherical twists are 
\begin{equation}\label{equation-covering-twists}
T_M = \CL_M^{-d}[-1]
\qquad\text{and}\qquad
T_X = \tau \circ \CL_X^{-d}[1].
\end{equation} 
Since $\tau(\CL_X) \cong \CL_X$ it follows that the twist $T_X$ commutes with $\CL_X$ and
\begin{equation}\label{equation-covering-rs}
\rho = \tau \circ \CL_X^{-d}[1] \circ \CL_X^d = \tau[1],
\qquad 
\sigma = \CL_X^{d} \circ f^*\omega_M[\dim M] \circ \tau \circ \CL_X^{-d}[1] \circ \CL_X^m = \tau \circ f^*(\omega_M \otimes \CL_M^m) [\dim M + 1],
\end{equation}
so with our assumptions $\rho^2$ and $\sigma^2$ are shifts.

\bigskip

Finally, assume that $\tilde{f}:\widetilde{X} \to \TM$ and $f:X \to M$ are as in Example~\ref{example-root},
so that $\BD(X) = \BD(\TX)^{\mu_2}$ and $\BD(M) = \BD(\TM)^{\mu_2}$ are the $\mu_2$-equivariant derived categories of $\TX$ and $\TM$ respectively.
The functors $f^*$, $f_*$, and $f^!$ can be thought of as $\tf^*$, $\tf_*$, and $\tf^!$ with their natural equivariant structures (see~\cite{kuznetsov2014derived} for details).
Denote $\CL_\TX := \tf^*\CL_\TM$ and let $\chi$ be the nontrivial character of $\mu_2$ (so that $\chi^2 = 1$). Note that equivariantly we have
\begin{equation*}
f_*\CO_X = (\CL_M^{-d}\otimes \chi) \oplus \CO_M
\qquad\text{and}\qquad
\omega_{X/M} = \CL_M^d \otimes \chi.
\end{equation*}
Therefore, analogously to the previous case we have
\begin{multline*}
f^!(f_*(f^*(F))) \cong 
f^!(F \otimes f_*\CO_X) \cong 
f^!(F \otimes (\CL_M^{-d} \otimes \chi \oplus \CO_M)) \cong \\ \cong
f^*(F) \otimes (\CL_X^{-d} \otimes \chi \oplus \CO_X) \otimes \CL_X^d \otimes \chi \cong
f^*(F) \oplus f^!(F).
\end{multline*}
Computing analogously the composition $f^*\circ f_* \circ f^!$ we see that Definition~\ref{defsph} holds, so $f_*$ is a spherical functor. 
Finally, the standard distinguished triangles
\begin{equation*}
F \to f_*f^*(F) \to F \otimes \CL_M^{-d} \otimes \chi
\qquad\text{and}\qquad
F \otimes \CL_X^{-d} \otimes \chi \to f^*f_*(F) \to F
\end{equation*}
(note that $\tau$ acts trivially on any equivariant sheaf), show that in this case the spherical twists are 
\begin{equation}\label{equation-root-twists}
T_M = \CL_M^{-d} \otimes \chi [-1]
\qquad\text{and}\qquad
T_X = \CL_X^{-d} \otimes \chi[1].
\end{equation} 
Clearly, $T_X$ commutes with $\CL_X$ and 
\begin{equation}\label{equation-root-rs}
\rho = \chi \circ \CL_X^{-d}[1] \circ \CL_X^d = \chi[1],
\qquad 
\sigma = \CL_X^{d}\circ\chi\circ f^*\omega_M[\dim M] \circ \chi \circ \CL_X^{-d}[1] \circ \CL_X^m = f^*(\omega_M\otimes\CL_M^m)[\dim M + 1],
\end{equation}
so with our assumptions $\rho^2$ and $\sigma$ are shifts.
This finishes the proof.
\end{proof}

Now we return to the abstract situation of a spherical functor $\Phi:\BD(X) \to \BD(M)$ with the corresponding
spherical twists $T_X$ and $T_M$. We consider the following autoequivalences of $\BD(X)$
\begin{align}
\rho &:= T_X\circ \CL_X^d,\label{equation-sigma}\\
\sigma &:= \SSS_X \circ T_X\circ \CL_X^m.\label{equation-tau}
\end{align}

\begin{theorem}\label{theorem-fcy}
Assume that $M$ and $X$ are smooth projective varieties (or stacks) with a spherical functor $\Phi:\BD(X) \to \BD(M)$
between their derived categories. Let $T_M$ and $T_X$ be the spherical twists. 
Assume that $\BD(M)$ has a rectangular Lefschetz decomposition
\begin{equation}\label{lefy}
\BD(M) = \langle \CB, \CB\otimes\CL_M, \dots, \CB\otimes\CL_M^{m-1} \rangle.
\end{equation} 
Assume that there is some $1 \le d < m$ such that for all $i \in \ZZ$ we have
\begin{equation}\label{equation-tm-cb}
T_M(\CB \otimes \CL_M^i) = \CB \otimes \CL_M^{i-d}.
\end{equation} 
Assume further that there is a line bundle $\CL_X$ on $X$ 
such that $\Phi$ intertwines between $\CL_X$ and $\CL_M$ twists:
\begin{equation}\label{intertwine}
\CL_M \circ \Phi \cong \Phi \circ \CL_X.
\end{equation} 
\item 
Finally, assume that the twist $T_X$ commutes with $\CL_X$ 
\begin{equation}\label{equation-tx-lx}
T_X\circ \CL_X = \CL_X \circ T_X.
\end{equation} 

Then the functor $\Phi^*:\BD(M) \to \BD(X)$ is fully faithful on the component $\CB$ of $\BD(M)$ and induces a semiorthogonal decomposition
\begin{equation}\label{sodx}
\BD(X) = \lan \CA_X,\CB_X,\CB_X\otimes\CL_X, \dots, \CB_X\otimes \CL_X^{m-d-1} \ran,
\end{equation}
where $\CB_X = \Phi^*(\CB)$ and $\CA_X$ is the orthogonal subcategory. 
Moreover, if
$c = \gcd(d,m)$
then $d/c$ power of the Serre functor of the category $\CA_X$ can be expressed as
\begin{equation*}
\SSS_{\CA_X}^{d/c} \cong \rho^{-m/c} \circ \sigma^{d/c}.
\end{equation*}
In particular, if some powers of $\rho$ and $\sigma$ are shifts then
$\CA_X$ is a fractional Calabi--Yau category.
\end{theorem}

\begin{remark}
If $d = m$ then the functor $\Phi^*$ is not fully faithful, but still for $\CA_X = \BD(X)$ the result of the Theorem holds.
Indeed, by~\eqref{equation-sigma} and~\eqref{equation-tau} we have $\SSS_{\CA_X} = \SSS_X = \rho^{-1}\circ \sigma$
which agrees with the formula in the Theorem since in this case $c = d = m$.
\end{remark}

Note that Examples~\ref{example-divisor}, \ref{example-covering} and~\ref{example-root} satisfy the assumptions
\eqref{equation-tm-cb}, \eqref{intertwine}, and \eqref{equation-tx-lx} of the Theorem
(in the last Example we need to assume additionally that $\CB \otimes \chi = \CB$, i.e.\ that the Lefschetz decomposition
is induced by a Lefschetz decomposition of $\BD(\widetilde{M})$).
Indeed, \eqref{intertwine} is given by the projection formula, \eqref{equation-tm-cb} and \eqref{equation-tx-lx} 
follow from the description of the functors $T_M$ and $T_X$ in~\eqref{equation-divisor-twists}, \eqref{equation-covering-twists}, and~\eqref{equation-root-twists}.
Note also that if $\omega_M = \CL_M^{-m}$ then
the functors $\rho$ and $\sigma$ are shifts in the first example, and their squares are shifts in the second and the third examples
as it was observed in the proof of Proposition~\ref{proposition-examples}.
Thus, in all these cases the constructed category $\CA_X$ is a fractional Calabi--Yau category. Below we rewrite the conclusion
of the Theorem (assuming a Lefschetz decomposition~\eqref{lefy} of $\BD(M)$ is given and $\omega_M = \CL_M^{-m}$)
in all three examples explicitly, substituting the expressions~\eqref{equation-divisor-rs}, \eqref{equation-covering-rs}, and~\eqref{equation-root-rs}
into the general formula.

\begin{corollary}\label{corollary-divisor}
If $f: X \to M$ is as in Example~\ref{example-divisor}, then $\SSS_{\CA_X}^{d/c} = [(\dim M + 1)d/c - 2m/c]$.
\end{corollary}

\begin{corollary}\label{corollary-covering}
If $f: X \to M$ is as in Example~\ref{example-covering}, then $\SSS_{\CA_X}^{d/c} = \tau^{(m-d)/c} [(\dim M + 1)d/c - m/c]$.
\end{corollary}

\begin{corollary}\label{corollary-root}
If $f: X \to M$ is as in Example~\ref{example-root}, then 
$\SSS_{\CA_X}^{d/c} = \chi^{m/c}[(\dim M + 1)d/c - m/c]$.
\end{corollary}

We do not know whether there are other examples of spherical functors for which the assumptions
of the Theorem are satisfied. Of course, it is tempting to replace the double cover example with 
a cyclic cover of arbitrary degree $k$, but the corresponding pushforward functor is not spherical,
so the Theorem does not apply in this case.
However, as Alex Perry notes, they are so-called $\PP^{k-1}$-functors, so it may well be that 
a generalization of our construction does something in this case as well.

\subsection{The induced semiorthogonal decomposition}\label{subsection-induced-sod}

We start with the first part of the Theorem (full faithfulness and a semiorthogonal decomposition).
This result in fact is quite simple. Moreover, for this to be true we do not need
to know that the Lefschetz collection in $\BD(Y)$ generates the whole category. So we state
here a slightly more general result.

\begin{lemma}\label{indlcx}
Assume that $\CB \subset \BD(M)$ is an admissible subcategory,
\begin{equation}\label{lcy}
\langle \CB, \CB\otimes\CL_M,\dots, \CB\otimes\CL_M^{m-1} \rangle \subset \BD(M)
\end{equation} 
is a rectangular Lefschetz collection and $\Phi:\BD(X) \to \BD(M)$ is a spherical functor
such that~\eqref{equation-tm-cb} and~\eqref{intertwine} hold. 
Then the functor $\Phi^*_{|\CB}:\CB \to \BD(X)$ is fully faithful and, denoting $\CB_X := \Phi^*(\CB)$,
the sequence of subcategories $\CB_{X}, \CB_{X}\otimes\CL_X, \dots, \CB_{X}\otimes \CL_X^{m-d-1}$
is semiorthogonal and extends to the semiorthogonal decomposition~\eqref{sodx} of $\BD(X)$.
\end{lemma}
\begin{proof}
Denote the embedding functor $\CB \to \BD(M)$ by $\beta_M$. The category $\CB$ is admissible, hence 
$\beta_M$ has a right adjoint which we denote by $\beta_M^!:\BD(M) \to \CB$. 
Therefore the functor $\Phi^*\circ\beta_M:\CB \to \BD(X)$ also has a right adjoint $\beta_M^!\circ \Phi$. We want to show that the composition 
$\beta_M^!\circ \Phi \circ \Phi^* \circ \beta_M$ is the identity. For this we compose~\eqref{defty} with $\beta_M^!$ on the left and with~$\beta_M$ on the right:
\begin{equation*}
\beta_M^!\circ T_M \circ\beta_M \to \beta_M^!\circ \beta_M \to \beta_M^!\circ \Phi \circ \Phi^* \circ \beta_M.
\end{equation*}
Note that the functor in the middle is the identity of $\CB$ (since $\beta_M$ is fully faithful), so it is enough to check
that the functor on the left is zero. As the kernel of $\beta_M^!$ is the orthogonal ${\CB}^\perp$, it is enough to check 
that the image of $T_M \circ \beta_M$ is contained in this subcategory. But this image is $T_M(\CB)$ and by~\eqref{equation-tm-cb} 
it is in $\CB \otimes\CL_M^{-d} \subset \CB^\perp$ by the twist
\begin{equation*}
\langle \CB\otimes\CL_M^{1-m}, \dots, \CB\otimes\CL_M^{-1}, \CB \rangle
\end{equation*}
of~\eqref{lcy} as $1 \le d \le m-1$.

For the semiorthogonality we have to check that the composition of functors $\beta_M^!\circ \Phi \circ \CL_X^{-i} \circ \Phi^* \circ \beta_M$ 
is zero for each $1 \le i \le m-d-1$. For this we use the intertwining property~\eqref{intertwine} and rewrite this composition as 
$\beta_M^!\circ \CL_M^{-i} \circ \Phi \circ \Phi^* \circ \beta_M$. Then we compose~\eqref{defty} with $\beta_M^! \circ \CL_M^{-i}$ on the left and with~$\beta_M$ on the right:
\begin{equation*}
\beta_M^! \circ \CL_M^{-i} \circ T_M \circ \beta_M \to \beta_M^! \circ \CL_M^{-i} \circ \beta_M \to \beta_M^! \circ \CL_M^{-i} \circ \Phi \circ \Phi^* \circ \beta_M.
\end{equation*}
Clearly, $\Im(\CL_M^{-i}\circ\beta_M) = \CB \otimes \CL_M^{-i}$ and $\Im(\CL_M^{-i} \circ T_M \circ \beta_M) = \CB \otimes \CL_M^{-i-d}$,
so as both these categories are in~$\CB^\perp$, they are killed by $\beta_M^!$, hence the first two terms of the triangle are zero.
Hence so is the third.
As we already have checked the embedding functor of $\CB_X$ has a right adjoint, the subcategory is right admissible
and thus gives the required semiorthogonal decomposition.
\end{proof}

In what follows we denote by $\beta_X:\CB \to \BD(X)$ and $\beta_X^!:\BD(X) \to \CB$ the fully faithful
embedding constructed in Lemma~\ref{indlcx} and its right adjoint functor, so that
\begin{equation}\label{equation-beta-x-beta-m}
\beta_X = \Phi^* \circ \beta_M,
\qquad
\beta_X^! = \beta_M^! \circ \Phi,
\end{equation} 
and consider the constructed Lefschetz collection
\begin{equation}\label{lefx}
\langle \CB_{X}, \CB_{X}\otimes\CL_X, \dots, \CB_{X}\otimes \CL_X^{m-d-1} \rangle \subset \BD(X).
\end{equation}
Further we will need the following

\begin{lemma}\label{lemma-sigma-bx}
For the functors $\rho$ and $\sigma$ we have
\begin{equation*}
\rho \circ \Phi^* \cong \Phi^* \circ T_M \circ \CL_M^d[2]
\qquad\text{and}\qquad 
\sigma \circ \Phi^* = \Phi^*\circ\CL_M^m\circ\SSS_M[1].
\end{equation*}
In particular, all components of~\eqref{lefx} are preserved by $\rho$ and $\sigma$.
\end{lemma}
\begin{proof}
The first equality follows from the definition of $\rho$, assumption~\eqref{intertwine} and Corollary~\ref{corollary-t-phi}.
The second is checked similarly:
\begin{multline*}
\sigma \circ \Phi^* =
\SSS_X \circ T_X \circ \CL_X^m \circ \Phi^* \cong
\SSS_X \circ \Phi^* \circ T_M \circ \CL_M^m [2] \cong
\\ \cong
\SSS_X \circ \Phi^* \circ \SSS_M^{-1} \circ T_M \circ \CL_M^m \circ \SSS_M [2] \cong
\Phi^! \circ T_M \circ \CL_M^m \circ \SSS_M [2] \cong
\Phi^* \circ \CL_M^m \circ \SSS_M [1],
\end{multline*}
the first is the definition of $\sigma$, 
the second is~\eqref{intertwine} and Corollary~\ref{corollary-t-phi},
the third and the fourth is Lemma~\ref{lemma-serre-basics},
and the last is Corollary~\ref{corollary-t-phi} again. It remains to note that by~\eqref{equation-tm-cb}
\begin{equation*}
\rho(\CB_X) = 
(\rho \circ \Phi^*)(\CB) =
(\Phi^* \circ T_M \circ \CL_M^d)(\CB) = 
\Phi^*(\CB) =
\CB_X,
\end{equation*}
so $\rho$ preserves $\CB_X$. Since $\rho$ commutes with $\CL_X$ by~\eqref{equation-tx-lx}, it also preserves all the other components of~\eqref{lefx}.
An analogous argument (with Lemma~\ref{lemma:serre-lefschetz} used
instead of~\eqref{equation-tm-cb}) works for $\sigma$ 
(note that $\sigma$ commutes with~$\CL_X$ by~\eqref{equation-tx-lx} and Lemma~\ref{lemma-serre-basics}).
\end{proof}

We denote by $\CA_X$ the orthogonal of the collection~\eqref{lefx}:
\begin{equation}\label{defax}
\CA_X := \langle \CB_{X}, \CB_{X}\otimes\CL_X, \dots, \CB_{X}\otimes \CL_X^{m-d-1} \rangle^\perp \subset \BD(X).
\end{equation}
This gives the required semiorthogonal decomposition~\eqref{sodx}. It follows also that the category $\CA_X$ is preserved by $\rho$ and $\sigma$.
Sometimes the following alternative description of $\CA_X$ is useful.

\begin{lemma}\label{cata}
Let $\CA_X \subset \BD(X)$ be the subcategory defined by~\eqref{defax}. Then
\begin{equation*}
\CA_X = \{ F \in \BD(X)\ |\ \Phi(F) \in \langle \CB\otimes \CL_M^{-d}, \dots, \CB \otimes \CL_M^{-1} \rangle \subset \BD(M) \}.
\end{equation*}
\end{lemma}
\begin{proof}
By definition we have
\begin{equation*}
\CA_X = \{ F \in \BD(X)\ |\ \Hom(\Phi^*(\CB),F) = \dots = \Hom(\Phi^*(\CB\otimes\CL_M^{m-d-1}),F) = 0 \}.
\end{equation*}
By adjunction this can be rewritten as
\begin{equation*}
\CA_X = \{ F \in \BD(X)\ |\ \Hom(\CB,\Phi(F)) = \dots = \Hom(\CB\otimes\CL_M^{m-d-1},\Phi(F)) = 0 \}.
\end{equation*}
So, the result follows from the twist $\BD(M) = \langle \CB\otimes\CL_M^{-d}, \dots, \CB\otimes \CL_M^{-1}, \CB, \dots, \CB\otimes\CL_M^{m-d-1} \rangle$
of~\eqref{lefy}. 
\end{proof}

\subsection{Rotation functors}\label{subsection-rotation}

Now we already have proved the first part of the Theorem, so it remains to compute the Serre functor.
The main instruments for this are rotation functors.

In general, a {\sf rotation functor}
can be defined in a presence of a rectangular Lefschetz collection 
\begin{equation*}
\langle \CB, \CB\otimes\CL_Y,\dots, \CB\otimes\CL_Y^{s-1} \rangle \subset \BD(Y)
\end{equation*} 
on a smooth projective variety (or a stack) $Y$. It is defined
as the composition of the twist and the left mutation functors:
\begin{equation}\label{roty}
\SO_\CB := \LL_\CB \circ \CL_Y.
\end{equation}
The following straightforward observation is quite useful.

\begin{lemma}\label{rotation-power}
If 
$\langle \CB, \CB\otimes\CL_Y,\dots, \CB\otimes\CL_Y^{s-1} \rangle \subset \BD(Y)$
is a rectangular Lefschetz collection and $\SO_\CB$
is the corresponding rotation functor,
then
\begin{equation*}
(\SO_{\CB})^i = \LL_{\langle \CB,\CB\otimes\CL_Y,\dots,\CB\otimes\CL_Y^{i-1} \rangle} \circ \CL_Y^i
\end{equation*}
for any $0 \le i \le s$.
\end{lemma}
\begin{proof}
By Lemma~\ref{lemma-mutations-basics} we have
\begin{multline*}
(\SO_\CB)^i = 
\SO_\CB \circ \SO_\CB \circ \dots \circ \SO_\CB =
(\LL_\CB \circ \CL_Y) \circ (\LL_\CB \circ \CL_Y) \circ \dots \circ (\LL_\CB \circ \CL_Y) = \\
\LL_\CB \circ (\CL_Y \circ \LL_\CB \circ \CL_Y^{-1}) \circ (\CL_Y^2 \circ \LL_\CB \circ \CL_Y^{-2} ) \circ \dots \circ (\CL_Y^{i-1} \circ \LL_\CB \circ \CL_Y^{1-i}) \circ \CL_Y^i = \\
\LL_\CB \circ \LL_{\CB \otimes \CL_Y} \circ \LL_{\CB \otimes \CL_Y^2} \circ \dots \circ \LL_{\CB \otimes \CL_Y^{i-1}} \circ \CL_Y^i =
\LL_{\langle \CB, \CB\otimes \CL_Y, \dots, \CB \otimes \CL_Y^{i-1} \rangle} \circ \CL_Y^i,
\end{multline*}
and we are done.
\end{proof}

In what follows we will consider two rectangular Lefschetz collections: the first is~\eqref{lefy} generating $\BD(M)$,
and the second is~\eqref{lefx} (which is nonfull). We denote the corresponding rotation functors by $\SO_M$ and $\SO_X$.
So, by definition of mutation functors we have the following distinguished triangles:
\begin{equation}\label{def-om}
\beta_M\beta_M^!\CL_M \to \CL_M \to \SO_M
\end{equation} 
and
\begin{equation}\label{def-ox}
\beta_X\beta_X^!\CL_X \to \CL_X \to \SO_X.
\end{equation} 
It is easy to see that the functor $\SO_M$ is nilpotent.

\begin{corollary}\label{corollary-so-i-zero}
For each $0 \le i \le m$ the $i$-th power $\SO^i_M$ of the rotation functor vanishes on the subcategory $\langle \CB \otimes \CL_M^{-i}, \dots, \CB \otimes \CL_M^{-1} \rangle \subset \BD(M)$.
In particular, the $m$-th power of $\SO_M$ vanishes identically.
\end{corollary}
\begin{proof}
Indeed, the twist by $\CL_M^i$ takes the subcategory $\langle \CB \otimes \CL_M^{-i}, \dots, \CB \otimes \CL_M^{-1} \rangle$ 
to the subcategory $\langle \CB,\dots,\CB\otimes\CL_M^{i-1} \rangle$,
which is killed by the mutation functor $\LL_{\langle \CB,\dots,\CB\otimes\CL_M^{i-1} \rangle}$. 
Finally, for $i = m$ the subcategory $\langle \CB \otimes \CL_M^{-m}, \dots, \CB \otimes \CL_M^{-1} \rangle \subset \BD(M)$
equals $\BD(M)$ by~\eqref{lefy}.
\end{proof}

It is also easy to see that the functor $\SO_X$ commutes with $\rho$ and $\sigma$:

\begin{lemma}\label{lemma-sigma-ox}
We have $\rho \circ \SO_X \cong \SO_X \circ \rho$ and $\sigma \circ \SO_X \cong \SO_X \circ \sigma$.
\end{lemma}
\begin{proof}
Indeed, $\SO_X$ is the composition of $\CL_X$ with $\LL_{\CB_X}$. But $\CL_X$ commutes with $\rho$ and $\sigma$ by~\eqref{equation-tx-lx}
and Lemma~\ref{lemma-serre-basics}, and $\LL_{\CB_X}$ commutes with $\rho$ and $\sigma$ by Lemma~\ref{lemma-sigma-bx}.
\end{proof}

\subsection{The fundamental relation}\label{subsection-relation}

In a contrast to the nilpotency of $\SO_M$, the functor $\SO_X$ induces an autoequivalence of the subcategory $\CA_X$.
Moreover, its $d$-th power coincides on $\CA_X$ with the autoequivalence~$\rho$.
This follows from a careful investigation of the relation between the rotation functors $\SO_M$ and~$\SO_X$,
and in the end leads to the proof of the Theorem.

\begin{lemma}\label{lemma-phi-o-intertwine}
For any $0 \le i \le d-1$ there is a morphism of functors $\Phi^*\circ \SO_M^i \xrightarrow{\ \gamma^i\ } \SO_X^i \circ \Phi^*$ inducing an isomorphism 
\begin{equation*}
\Phi^*\circ \SO_M^i \cong \SO_X^i \circ \Phi^*
\end{equation*}
on the subcategory $\langle \CB \otimes \CL_M^{d-i}, \dots, \CB \otimes \CL_M^{d-1} \rangle^\perp = 
\langle \CB \otimes \CL_M^{d-m},\CB \otimes \CL_M^{d+1-m},\dots,\CB \otimes \CL_M^{d-i-1} \rangle \subset \BD(M)$.
\end{lemma}
\begin{proof}
For $i = 0$ there is nothing to prove, so consider the case $i = 1$. 
Then we have the following diagram
\begin{equation*}
\xymatrix@C=5em{
\Phi^* \beta_M \beta_M^! \CL_M \ar[r]^-{\epsi_{\beta_M,\beta_M^!}} \ar[d]_{\eta_{\Phi,\Phi^*}} & \Phi^* \CL_M \ar[r] \ar@{=}[d] & \Phi^* \SO_M \ar@{..>}[d] \\
\beta_X \beta_X^! \CL_X \Phi^* \ar[r]^-{\epsi_{\beta_X,\beta_X^!}} & \CL_X \Phi^* \ar[r] & \SO_X \Phi^* 
}
\end{equation*}
where 
the rows are obtained by compositng~\eqref{def-om} and~\eqref{def-ox} with $\Phi^*$,
the isomorphism in the middle column is induced by~\eqref{intertwine}, while the arrow in the left column is given by
the isomorphisms~\eqref{equation-beta-x-beta-m} 
and~\eqref{intertwine} 
(altogether giving an isomorphism $\beta_X \beta_X^! \CL_X \Phi^* \cong \Phi^* \beta_M \beta_M^! \CL_M \Phi \Phi^*$)
and the unit of the adjunction $\eta_{\Phi,\Phi^*}:\id \to \Phi\Phi^*$. The left square clearly commutes, hence it extends to a morphism of triangles
by the dotted arrow on the right which we denote by $\gamma$. It remains to show that $\gamma$ is an isomorphism on the subcategory 
$(\CB \otimes \CL_M^{d-1})^\perp =  \langle \CB \otimes \CL_M^{d-m}, \CB \otimes \CL_M^{d+1-m}, \dots, \CB \otimes \CL_M^{d-2} \rangle \subset \BD(M)$.

By construction of the left arrow in the diagram, the first column extends to a triangle
\begin{equation*}
\Phi^* \beta_M \beta_M^! \CL_M T_M \xrightarrow{\ \ \ } \Phi^* \beta_M \beta_M^! \CL_M \xrightarrow{\ \eta_{\Phi,\Phi^*}\ } \beta_X \beta_X^! \CL_X \Phi^*
\end{equation*}
Note that the first functor here vanishes on the subcategory $(\CB \otimes \CL_M^{d-1})^\perp \subset \BD(M)$.
Indeed, by~\eqref{equation-tm-cb} the functor $T_M$ takes it into $(\CB \otimes \CL_M^{-1})^\perp \subset \BD(M)$,
then $\CL_M$ takes it to
$\CB^\perp \subset \BD(M)$
which is 
killed by $\beta_M^!$.
It follows that the left arrow in the above diagram is an isomorphism on the subcategory $(\CB \otimes \CL_M^{d-1})^\perp$,
hence so is the right arrow.

Now assume that $i > 1$. We define the map $\Phi^* \circ \SO_M^i \to \SO_X^i \circ \Phi^*$ by an iteration of the map $\gamma$
\begin{equation*}
\Phi^*\circ \SO_M^i \xrightarrow{\ \gamma\ } 
\SO_X \circ \Phi^* \circ \SO_M^{i-1} \xrightarrow{\ \gamma\ } 
\dots \xrightarrow{\ \gamma\ } 
\SO_X^{i-1} \circ \Phi^* \circ \SO_M \xrightarrow{\ \gamma\ } 
\SO_X^i \circ \Phi^* 
\end{equation*}
and denote it by $\gamma^i$.
It remains to prove that it induces an isomorphism on the specified subcategory. 
We prove this by induction in $i$, the case $i = 1$ proved above being the base of the induction.
So, assume that we already have proved that $\gamma^{i-1}$ induces an isomorphism
\begin{equation*}
\Phi^*\circ \SO_M^{i-1} \cong \SO_X^{i-1} \circ \Phi^*
\end{equation*}
on $\langle \CB \otimes \CL_M^{d-i+1},\dots,\CB \otimes \CL_M^{d-1} \rangle^\perp \subset \BD(M)$. 
Assume now that $F \in \langle \CB \otimes \CL_M^{d-i},\dots,\CB \otimes \CL_M^{d-1} \rangle^\perp \subset \BD(M)$. 
Then by the induction hypothesis
\begin{equation*}
\SO_X^{i} \circ \Phi^*(F) = 
\SO_X(\SO_X^{i-1} \circ \Phi^*(F)) \cong
\SO_X(\Phi^*\circ \SO_M^{i-1}(F)) = 
(\SO_X \circ \Phi^*) (\SO_M^{i-1}(F)).
\end{equation*}
On the other hand, by Lemma~\ref{rotation-power} we have
$\SO_M^{i-1}(F) = \LL_{\langle \CB, \dots, \CB \otimes \CL_M^{i-2} \rangle}(F \otimes \CL_M^{i-1})$.
It is easy to see that $F \otimes \CL_M^{i-1} \in (\CB \otimes \CL_M^{d-1})^\perp$ and $\langle \CB, \dots, \CB \otimes \CL_M^{i-2} \rangle \subset (\CB \otimes \CL_M^{d-1})^\perp$ as well.
It follows from the definition of mutations that $\SO_M^{i-1}(F) = \LL_{\langle \CB, \dots, \CB \otimes \CL_M^{i-2} \rangle}(F \otimes \CL_M^{i-1}) \in (\CB \otimes \CL_M^{d-1})^\perp$, and hence
the base of induction applies and 
\begin{equation*}
(\SO_X \circ \Phi^*) (\SO_M^{i-1}(F)) \cong (\Phi^* \circ \SO_M)(\SO_M^{i-1}(F)) = (\Phi^* \circ \SO_M^i)(F).
\end{equation*}
This completes the proof of the Lemma.
\end{proof}

Consider the composition of maps $\Phi^* \circ \SO_M^i \circ \Phi \xrightarrow{\ \gamma^i\ } \SO_X^i \circ \Phi^* \circ \Phi \xrightarrow{\ \epsi_{\Phi^*,\Phi}\ } \SO_X^i$.

\begin{proposition}\label{proposition-triangle-o-o}
For each $0 \le i \le d$ there is a distinguished triangle of functors
\begin{equation*}
\Phi^*\circ\SO_M^i \circ \Phi \xrightarrow{\ \epsi_{\Phi^*,\Phi} \circ \gamma^i\ } \SO_X^i \xrightarrow{\qquad} T_X\circ\CL_X^i.
\end{equation*}
\end{proposition}
\begin{proof}
We prove this by induction in $i$. The base of the induction, the case $i = 0$, is provided by the triangle~\eqref{deftx}.
So, assume that $i > 0$. Consider the diagram
\begin{equation*}
\xymatrix@C=4em{
\Phi^*\circ \SO_M^{i-1} \circ \beta_M\beta_M^!\CL_M \circ \Phi \ar[r]^-{\epsi_{\beta_M,\beta_M^!}} \ar[d]^{\gamma^{i-1}} & 
\Phi^*\circ \SO_M^{i-1} \circ \CL_M \circ \Phi \ar[r] \ar[d]^{\gamma^{i-1}} & 
\Phi^*\circ \SO_M^{i-1} \circ \SO_M \circ \Phi \ar[d]^{\gamma^{i-1}}
\\
\SO_X^{i-1} \circ \Phi^* \circ \beta_M\beta_M^!\CL_M \circ \Phi \ar[r]^-{\epsi_{\beta_M,\beta_M^!}} \ar@{=}[d] & 
\SO_X^{i-1} \circ \Phi^* \circ \CL_M \circ \Phi \ar[r] \ar[d]^{\epsi_{\Phi^*,\Phi}} & 
\SO_X^{i-1} \circ \Phi^* \circ \SO_M \circ \Phi \ar[d]^{\epsi_{\Phi^*,\Phi}\circ \gamma}
\\
\SO_X^{i-1} \circ \Phi^* \circ \beta_M\beta_M^! \circ \Phi \circ \CL_X \ar[r]^-{\epsi_{\Phi^*,\Phi}\circ\epsi_{\beta_M,\beta_M^!}} & 
\SO_X^{i-1} \circ \CL_X \ar[r] & 
\SO_X^{i-1} \circ \SO_X 
}
\end{equation*}
Here the first row is obtained by composing the triangle~\eqref{def-om} with $\Phi^*\circ \SO_M^{i-1}$ on the left and $\Phi$ on the right,
the second row is obtained by composing it with $\SO_X^{i-1} \circ \Phi^*$ on the left and $\Phi$ on the right,
and the last row is obtained by composing the triangle~\eqref{def-ox} with $\SO_X^{i-1}$ on the left
(taking into account~\eqref{equation-beta-x-beta-m}). So the rows are distinguished triangles
and the vertical maps form morphisms of distinguished triangles (for the first this is evident,
and for the second this follows from the definition of $\gamma$ in Lemma~\ref{lemma-phi-o-intertwine}).
Composing the morphisms of these triangles, we get the following commutative diagram
\begin{equation*}
\xymatrix@C=4em{
\Phi^*\circ \SO_M^{i-1} \circ \beta_M\beta_M^!\CL_M \circ \Phi \ar[r]^-{\epsi_{\beta_M,\beta_M^!}} \ar[d]^{\gamma^{i-1}} & 
\Phi^*\circ \SO_M^{i-1} \circ \CL_M \circ \Phi \ar[r] \ar[d]^{\epsi_{\Phi^*,\Phi} \circ\gamma^{i-1}} & 
\Phi^*\circ \SO_M^{i-1} \circ \SO_M \circ \Phi \ar[d]^{\epsi_{\Phi^*,\Phi} \circ \gamma^i} \\
\SO_X^{i-1} \circ \beta_X\beta_X^!\CL_X \ar[r]^-{\epsi_{\beta_X,\beta_X^!}} & 
\SO_X^{i-1} \circ \CL_X \ar[r] &
\SO_X^{i-1} \circ \SO_X 
}
\end{equation*}
(we have rewritten the first term of the bottom row via~\eqref{equation-beta-x-beta-m}).

Note that $i \le d$ implies $\Im \beta_M = \CB \subset \langle \CB \otimes \CL_M^{d-(i-1)}, \dots, \CB \otimes \CL_M^{d-1} \rangle^\perp$ hence the left arrow
is an isomorphism by Lemma~\ref{lemma-phi-o-intertwine}. Moreover, by induction hypothesis the middle vertical map extends to a distinguished triangle
by $T_X \circ \CL_X^i$. Therefore, the octahedron axiom implies that the right vertical arrow extends to a distinguished triangle
\begin{equation*}
\Phi^*\circ \SO_M^{i} \circ \Phi \xrightarrow{\ \epsi_{\Phi^*,\Phi} \circ\gamma^i\ } \SO_X^i \xrightarrow{\qquad} T_X \circ \CL_X^i,
\end{equation*}
and thus proves the required claim.
\end{proof}

\begin{corollary}
The restriction of $\SO_X$ to the subcategory $\CA_X \subset \BD(X)$ is an autoequivalence such that 
\begin{equation}\label{equation-so-d}
\SO_{X|\CA_X}^d \cong \rho_{|\CA_X},
\end{equation} 
where $\rho$ is defined by~\eqref{equation-sigma}.
\end{corollary}
\begin{proof}
Let us restrict the triangle of Proposition~\ref{proposition-triangle-o-o} to $\CA_X$.
The first term of the triangle then vanishes by a combination of Lemma~\ref{cata}
and Corollary~\ref{corollary-so-i-zero}. Therefore, the functors given by the second and the third terms
are isomorphic, so it remains to use the definition~\eqref{equation-sigma} of $\rho$.
\end{proof}

\subsection{Proof of the Theorem}\label{subsection-proof}

To finish we need a relation between the Serre functor of $\CA_X$ and the rotation functor.

\begin{lemma}
The Serre functor of the category $\CA_X$ is given by
\begin{equation}\label{equation-sax}
\SSS^{-1}_{\CA_X} \cong \SO_X^{m-d} \circ \rho \circ \sigma^{-1}.
\end{equation} 
\end{lemma}
\begin{proof}
By Lemma~\ref{lemma-serre-mutation} we have $\SSS^{-1}_{\CA_X} = \LL_{\langle \CB_X, \dots, \CB_X \otimes \CL_X^{m-d-1} \rangle} \circ \SSS_X^{-1}$
and by~\eqref{equation-tau} we have $\SSS_X^{-1} = \CL_X^m \circ T_X \circ \sigma^{-1}$. Combining this we obtain
\begin{multline*}
\SSS^{-1}_{\CA_X} \cong 
\LL_{\langle \CB_X, \dots, \CB_X \otimes \CL_X^{m-d-1} \rangle} \circ \CL_X^{m} \circ T_X \circ \sigma^{-1}
\cong \\ \cong
(\LL_{\langle \CB_X, \dots, \CB_X \otimes \CL_X^{m-d-1} \rangle} \circ \CL_X^{m-d}) \circ (\CL_X^d \circ T_X) \circ \sigma^{-1}
\cong
\SO_X^{m-d} \circ \rho \circ \sigma^{-1}.
\end{multline*}
Here 
the last isomorphism is Lemma~\ref{rotation-power}.
\end{proof}

To finish the proof note that $\rho$ and $\sigma$ commute. Indeed, both are combinations of $T_X$, $\CL_X$, and $\SSS_X$,
but $T_X$ and $\CL_X$ commute by~\eqref{equation-tx-lx}, and $\SSS_X$ commutes with any autoequivalence by Lemma~\ref{lemma-serre-basics}.
Moreover, both $\rho$ and $\sigma$ commute with $\SO_X$ by Lemma~\ref{lemma-sigma-ox}. Therefore, taking the $d/c$ power of~\eqref{equation-sax},
where $c = \gcd(d,m)$, we obtain
\begin{equation*}
\SSS^{-d/c}_{\CA_X} \cong \SO_X^{d(m-d)/c} \circ \rho^{d/c} \circ \sigma^{-d/c}.
\end{equation*}
But $\SO_X^{d(m-d)/c} \cong \rho^{(m-d)/c}$ on the subcategory $\CA_X$ by \eqref{equation-so-d}, hence
$\SSS^{-d/c}_{\CA_X} \cong \rho^{m/c} \circ \sigma^{-d/c}$.
This completes the proof of the Theorem.

\begin{remark}
Note that we could weaken the assumptions of the Theorem as follows. First, we could replace $\CL_M$ and $\CL_X$
by arbitrary autoequivalences (not necessarily tensoring with a line bundle). Second, we could replace $\BD(X)$ and $\BD(M)$
by admissible subcategories (in other words, we could let $X$ and $M$ to be noncommutative varieties). The same proof
would apply in this larger generality. However, we do not know whether there are interesting examples of this more 
general situation.
\end{remark}

\begin{remark}\label{remark-non-rectangular}
Most of results of this section generalize to the case when the initial Lefschetz decomposition of $\BD(M)$ is not rectangular:
\begin{equation*}
\BD(M) = \langle \CB_0, \CB_1 \otimes \CL_M, \dots, \CB_{m-1} \otimes \CL_M^{m-1} \rangle,
\qquad \CB_0 \supset \CB_1 \supset \dots \supset \CB_{m-1}.
\end{equation*}
Let us list the necessary modifications to the claims and leave the reader to check that the same proofs work.
First, in Lemma~\ref{indlcx} the functor $\Phi^*$ is fully faithful only on components $\CB_i$ with $i \ge d$, 
and the induced semiorthogonal decomposition of $\BD(X)$ looks as
\begin{equation*}
\BD(X) = \langle \CA_X, \CB_{X, d} \otimes \CL^d_M, \dots, \CB_{X, m-1} \otimes \CL_M^{m-d-1} \rangle,
\end{equation*}
Further, instead of one rotation functor there is a sequence of functors, one for each component of the Lefschetz collection.
So, we have
\begin{equation*}
\SO_{M,i} = \LL_{\CB_i} \circ \CL_M
\qquad\text{and}\qquad
\SO_{X,i} = \LL_{\CB_{X, i}} \circ \CL_X
\end{equation*}
and instead of powers it is natural to consider products of sequences of these functor. So, for any $a \le b$ we define
\begin{equation*}
\SO_M^{[a,b]} = \SO_{M,a} \circ \SO_{M,a+1} \circ \dots \circ \SO_{M,b} 
\qquad\text{and}\qquad
\SO_X^{[a,b]} = \SO_{X,a} \circ \SO_{X,a+1} \circ \dots \circ \SO_{X,b}
\end{equation*}
Then $\SO_M^{[k,k+s-1]}$ vanishes on the subcategory $\langle \CB_k \otimes \CL_M^{-s}, \dots, \CB_{k+s-1} \otimes \CL_M^{-1} \rangle$
for any $0 \le k < k + s \le m$ and as a consequence $\SO_M^{[0,d-1]}\circ \Phi$ vanishes on $\CA_X$. Furthermore, 
a modification of Lemma~\ref{lemma-phi-o-intertwine} says that
if $\langle \CB_{k+1}\otimes \CL_M, \dots, \CB_{k+s-1} \otimes \CL_M^{s-1} \rangle \subset (\CB_k \otimes \CL_M^d)^\perp$ then
$\Phi^* \circ \SO_M^{[k,k+s-1]} \cong \SO_X^{[k, k+s-1]} \circ \Phi^*$ on the subcategory $\langle \CB_{k}\otimes \CL_M^{d-s}, \dots, \CB_{k+s-1} \otimes \CL_M^{d-1} \rangle$.
Finally, a modification of Proposition~\ref{proposition-triangle-o-o} says that if
\begin{equation}\label{equation-strange-condition}
\CB_{k+i} \otimes \CL_M^i \in \langle \CB_k \otimes \CL_M^d, \CB_{k+1} \otimes \CL_M^{d+1}, \dots, \CB_{k+s-1} \otimes \CL_M^{d+s-1} \rangle^\perp
\qquad\text{for all $1 \le i \le s$,}
\end{equation}
then there is a distinguished triangle
\begin{equation*}
\Phi^* \circ \SO_M^{[k,k+s-1]} \circ \Phi \to \SO_X^{[k,k+s-1]} \to T_X\CL_X^{s+1}.
\end{equation*}
In particular, if~\eqref{equation-strange-condition} holds for $k = 0$ and $s = d$ (which is definitely true if $\CB_0 = \CB_{d-1}$) then 
\begin{equation*}
\SO_X^{[0,d-1]} \cong \rho
\qquad\text{on $\CA_X$.}
\end{equation*}
On the other hand, the (inverse) Serre functor of $\CA_X$ can be expressed as $\SSS_{\CA_X}^{-1} =  \SO_X^{[d,m-1]} \circ \rho\circ \sigma^{-1}$ and the problem of 
generalizing the construction to this setup is in relating $\SO_X^{[d,m-1]}$ to $\SO_X^{[0,d-1]}$.
\end{remark}

\section{Explicit examples}\label{section-examples}

\subsection{Varieties with a rectangular Lefschetz decomposition}\label{subsection-lefschetz}

In fact, any variety $M$ has a rectangular Lefschetz decomposition of length $m = 1$ 
with respect to the anticanonical line bundle. However, this decomposition does not produce 
an interesting Calabi--Yau category as in this case $d = m$ and $\CA_X = \BD(X)$.

So, one can get something interesting only from a rectangular Lefschetz decomposition of length greater than 1.
In the following list we give a number of such decompositions. In most of these the line bundle $\CL_M$
is the ample generator of the Picard group, so we always assume this is the case unless something else is specified.
Moreover, in all these cases $\omega_M = \CL_M^{-m}$.

\begin{enumerate}
\item\label{list-pn} 
A projective space $\PP^n$ has a rectangular Lefschetz decomposition
\begin{equation*}
\BD(\PP^n) = \langle \CO_{\PP^n}, \CO_{\PP^n}(1), \dots, \CO_{\PP^n}(n) \rangle
\end{equation*}
of length $m = n + 1$.
\item\label{list-wpn} 
A weighted projective space $\PP(w_0,w_1,\dots,w_n)$ considered as a smooth toric stack
has a rectangular Lefschetz decomposition
\begin{equation*}
\BD(\PP(w_0,w_1,\dots,w_n)) = \langle \CO_{\PP(w_0,w_1,\dots,w_n)}, \CO_{\PP(w_0,w_1,\dots,w_n)}(1), \dots, \CO_{\PP(w_0,w_1,\dots,w_n)}(m-1) \rangle
\end{equation*}
of length $m = w := w_0+w_1+\dots+w_n$.
\item\label{list-quadric}
A smooth quadric of dimension $n = 4s + 2$ has a rectangular Lefschetz decomposition 
\begin{equation*}
\BD(Q^{4s+2}) = \langle \CB, \CB(2s+1) \rangle
\end{equation*}
of length $m = 2$ with respect to the line bundle $\CO(2s+1)$, where
\begin{equation*}
\CB = \langle \CO, \CO(1), \dots, \CO(2s), \CS(2s) \rangle
\end{equation*}
with $\CS$ being one of the two spinor bundles.
\item\label{list-gr} 
A Grassmannian $\Gr(k,n)$ with $(k,n)$ coprime has a rectangular Lefschetz decomposition 
\begin{equation*}
\BD(\Gr(k,n)) = \langle \CB, \CB(1), \dots, \CB(n-1) \rangle
\end{equation*}
of length $m = n$, with the category $\CB$ generated by the exceptional collection formed by the Schur functors
$\Sigma^\alpha\CU^\vee$, where $\CU$ is the tautological rank $k$ subbundle and $\alpha$ runs through the set
of all Young diagrams with at most $k-1$ rows and with $p$-th row of length at most $(n-k)(k-p)/k$:
\begin{equation*}
\CB = \langle \Sigma^\alpha\CU^\vee \mid \alpha_1 < (n-k)(k-1)/k,\ \alpha_2 < (n-k)(k-2)/k,\ \dots,\ \alpha_{k-1} < (n-k)/k \rangle,
\end{equation*}
see~\cite{fonarev2013minimal}.
\item\label{list-ogr} 
An orthogonal Grassmannian $\OGr(2,2n+1)$ has a rectangular Lefschetz decomposition 
\begin{equation*}
\BD(\OGr(2,2n+1)) = \langle \CB, \CB(1), \dots, \CB(2n-3) \rangle
\end{equation*}
of length $m = 2n-2$, with the category $\CB$ generated by the exceptional collection 
formed by symmetric powers of the dual tautological bundle and the spinor bundle:
\begin{equation*}
\CB = \langle \CO, \CU^\vee, \dots, S^{n-2}\CU^\vee, \CS \rangle,
\end{equation*}
see~\cite{kuznetsov2008exceptional}.
\item\label{list-other-homogeneous} 
Some other homogeneous spaces: some symplectic Grassmannians, e.g.
\begin{equation*}
\BD(\SGr(3,6)) = \langle \CB, \CB(1), \CB(2), \CB(3) \rangle, 
\qquad\text{where $\CB = \langle \CO, \CU^\vee \rangle$};
\end{equation*}
some (connected components of) orthogonal Grassmannians, e.g.
\begin{equation*}
\BD(\OGr_+(5,10)) = \langle \CB, \CB(1), \dots, \CB(7) \rangle, 
\qquad\text{where $\CB = \langle \CO, \CU^\vee \rangle$};
\end{equation*}
the Grassmannian of the simple group of type $G_2$ (the highest weight orbit in the projectivization of the adjoint representation)
\begin{equation*}
\BD(\GTGr) = \langle \CB, \CB(1), \CB(2) \rangle, 
\qquad\text{where $\CB = \langle \CO, \CU^\vee \rangle$},
\end{equation*}
where $\CU$ is the restriction of the tautological bundle under the natural embedding $\GTGr \hookrightarrow \Gr(2,7)$, see~\cite{kuznetsov2006hyperplane}.
\item\label{list-other-non-homogeneous}
Some quasihomogeneous spaces, e.g.\ a hyperplane section of $\Gr(2,2n+1)$:
\begin{equation*}
\BD(\IGr(2,2n+1)) = \langle \CB, \CB(1), \dots, \CB(2n-1) \rangle,
\qquad\text{where $\CB = \langle \CO, \CU^\vee, \dots, S^{n-1}\CU^\vee \rangle$},
\end{equation*}
see~\cite{kuznetsov2008exceptional}.
\end{enumerate}

One can also consider relative versions of the above decompositions. For example, if $\CE$ is a vector bundle
on a scheme $S$ then its projectivization $\PP_S(\CE)$ has a rectangular Lefschetz decomposition of length
equal to the rank of $\CE$ with the components equivalent to $\BD(S)$.

In general, given a minimal homogeneous space $M = G/P$ (i.e.\ with semisimple $G$ and maximal parabolic $P$)
it is expected that $\BD(M)$ has a rectangular Lefschetz decomposition as soon as the Euler characteristic of $M$
(which is equal to the rank of the Grothendieck group of $\BD(M)$ and which can be computed as the index 
of the Weyl group of $P$ in the Weyl group of $G$) is divisible by the index of $M$. For instance, 
it should exist on $\SGr(3,6n)$ and $\SGr(3,6n+4)$ for any $n$, and many others.

In some cases, when the rank of the Grothendieck group of such $M$ is not divisible by the index $i_M$, 
but they have a nontrivial common divisor $m$, it may be that there is a rectangular Lefschetz 
decomposition of length $m$ with respect to $\CO(i_M/m)$. For instance, for an even dimensional quadric $Q^{2k}$ the rank 
of the Grothendieck group is $2k+2$, while the index is $2k$, so the only nontrivial common divisor
is $2$. And indeed, if $k$ is odd $\BD(Q^{2k})$ admits a rectangular Lefschetz decomposition of length $2$ with respect to $\CO(k)$
(see case~\eqref{list-quadric} of the above list). However for even~$k$ it seems that there is no analogue for this decomposition.

Another example of this sort is $\Gr(2,6)$, when the rank of the Grothendieck group is $15$ and the index is $6$, so one can take $m = 3$,
and indeed there is a rectangular Lefschetz decomposition 
\begin{equation*}
\BD(\Gr(2,6)) = \langle \CB, \CB(2), \CB(4) \rangle,
\qquad\text{where $\CB = \langle \CO, \CU^\vee, S^2\CU^\vee, \CO(1), \CU^\vee(1) \rangle$.}
\end{equation*}

\subsection{Hypersurfaces}\label{subsection-hypersurfaces}

In this section we give explicit statements of Theorem~\ref{theorem-fcy} for hypersurfaces 
in some varieties with rectangular Lefschetz decompositions.
The first result in fact can be found in~\cite{kuznetsov2004derived}.

\begin{corollary}\label{corollary-divisor-pn}
Let $X \subset \PP^n$ be a smooth hypersurface of degree $d \le n + 1$ and $c = \gcd(d,n+1)$. 
The derived category of $X$ has a semiorthogonal decomposition 
\begin{equation*}
\BD(X) = \langle \CA_X, \CO_X, \dots, \CO_X(n-d) \rangle
\end{equation*}
and the Serre functor of $\CA_X$ has the property $\SSS_{\CA_X}^{d/c} = [(n+1)(d-2)/c]$. In particular,
if $d$ divides $n+1$ then $\CA_X$ is a Calabi--Yau category of dimension $(n+1)(d-2)/d$.
\end{corollary}

The most famous of these cases is that of a cubic fourfold (see~\cite{kuznetsov2010derived}),
when the category $\CA_X$ can be thought of as a noncommutative K3 surface. The case of a cubic
hypersurface of dimension 7 (when $\CA_X$ is a 3-Calabi--Yau category) was discussed in~\cite{iliev2011fano}.

\begin{corollary}\label{corollary-divisor-wpn}
Let $X \subset \PP(w_0,w_1,\dots,w_n)$ be a smooth hypersurface of degree $d \le w := \sum w_i$ in a weighted 
projective space and $c = \gcd(d,w)$. 
The derived category of $X$ has a semiorthogonal decomposition 
\begin{equation*}
\BD(X) = \langle \CA_X, \CO_X, \dots, \CO_X(w-d-1) \rangle
\end{equation*}
and the Serre functor of $\CA_X$ has the property $\SSS_{\CA_X}^{d/c} = [((n+1)d - 2w)/c]$. In particular,
if $d$ divides $w$ then $\CA_X$ is a Calabi--Yau category of dimension $n+1 -2w/d$.
\end{corollary}

\begin{corollary}\label{corollary-divisor-quadric}
Let $X \subset Q^{4s+2}$ be a hypersurface of degree $2s+1$ (thus $X$ is a complete intersection of type $(2,2s+1)$ in $\PP^{4s+3}$).
The derived category of $X$ has a semiorthogonal decomposition 
\begin{equation*}
\BD(X) = \langle \CA_X, \CO_X, \CO_X(1), \dots, \CO_X(2s), \CS(2s)_{|X} \rangle
\end{equation*}
and 
$\CA_X$ is a Calabi--Yau category of dimension $4s-1$.
\end{corollary}

The case $s = 1$ appeared in~\cite{iliev2011fano}.

\begin{corollary}\label{corollary-divisor-grkn}
Assume $\gcd(k,n) = 1$ and let $X \subset \Gr(k,n)$ be a hypersurface of degree $d \le n$ and $c = \gcd(d,n)$. 
The derived category of $X$ has a semiorthogonal decomposition 
\begin{equation*}
\BD(X) = \langle \CA_X, \CB_X, \CB_X(1), \dots, \CB_X(n-d-1) \rangle,
\end{equation*}
where the category $\CB$ is described in part~\eqref{list-gr} of Section~\ref{subsection-lefschetz}.
The Serre functor of $\CA_X$ has the property $\SSS_{\CA_X}^{d/c} = [(k(n-k)+1)d/c-2n/c]$. 
In particular, if $d$ divides $n$ then $\CA_X$ is a Calabi--Yau category of dimension $k(n-k)+1-2n/d$.
\end{corollary}

\begin{corollary}\label{corollary-divisor-ogr2}
Let $X \subset \OGr(2,2n+1)$ be a hypersurface of degree $d \le 2n-2$ and $c = \gcd(d,2n-2)$. 
The derived category of $X$ has a semiorthogonal decomposition 
\begin{equation*}
\BD(X) = \langle \CA_X, \CB_X, \CB_X(1), \dots, \CB_X(2n-3-d) \rangle,
\end{equation*}
where the category $\CB$ is described in part~\eqref{list-ogr} of Section~\ref{subsection-lefschetz}.
The Serre functor of $\CA_X$ has the property $\SSS_{\CA_X}^{d/c} = [4(n-1)(d-1)/c]$. 
In particular, if $d$ divides $2n-2$ then $\CA_X$ is a Calabi--Yau category of dimension $4(n-1)(d-1)/d$.
\end{corollary}

We leave the reader to formulate analogous results in other cases.

\subsection{Double coverings}\label{subsection-covers}

Here we restrict to stating what happens for double covers of projective spaces and Grassmannians.
The reader is welcome to formulate the other results.

\begin{corollary}\label{corollary-cover-pn}
Let $X \to \PP^n$ be a double covering ramified in a smooth hypersurface of degree $2d$ with $d \le n + 1$ and let $c = \gcd(d,n+1)$. 
Let $\tau$ be the involution of the double covering.
The derived category of $X$ has a semiorthogonal decomposition 
\begin{equation*}
\BD(X) = \langle \CA_X, \CO_X, \dots, \CO_X(n-d) \rangle
\end{equation*}
and the Serre functor of $\CA_X$ has the property $\SSS_{\CA_X}^{d/c} = \tau^{(n+1-d)/c}[(n+1)(d-1)/c]$. In particular,
if $d$ divides $n+1$ and $(n+1)/d$ is odd then $\CA_X$ is a Calabi--Yau category of dimension $(n+1)(d-1)/d$.
\end{corollary}

The case $n = 5$, $d = 2$ appeared in~\cite{iliev2011fano}.

\begin{corollary}\label{corollary-cover-grkn}
Assume that $\gcd(k,n) = 1$ and let $X \to \Gr(k,n)$ be a double covering ramified in a smooth hypersurface of degree $2d$ with $d \le n$ and let $c = \gcd(d,n)$. 
Let $\tau$ be the involution of the double covering.
The derived category of $X$ has a semiorthogonal decomposition 
\begin{equation*}
\BD(X) = \langle \CA_X, \CB_X, \dots, \CB_X(n-d-1) \rangle
\end{equation*}
where the category $\CB$ is described in part~\eqref{list-gr} of Section~\ref{subsection-lefschetz},
and the Serre functor of $\CA_X$ has the property $\SSS_{\CA_X}^{d/c} = \tau^{(n-d)/c}[(k(n-k)+1)d/c - n/c]$. In particular,
if $d$ divides $n$ and $n/d$ is odd then $\CA_X$ is a Calabi--Yau category of dimension $k(n-k) + 1 - n/d$.
\end{corollary}

One of the interesting cases here is formed by double covers of $\Gr(2,5)$ (i.e\ $k = 2$, $n = 5$, $d = 1$),
known as Gushel--Mukai 6-folds. See~\cite{kuznetsov2014derived,kuznetsov2015gushel} for more details.

\subsection{K3 categories}\label{subsection-k3}

Let us list the cases when the category $\CA_X$ is a 2-Calabi--Yau category:

\begin{itemize}
\item a cubic fourfold $X_3 \subset \PP^5$;
\item a hyperplane section $X_1 \subset \Gr(3,10)$ (Debarre--Voisin varieties, see~\cite{debarre2010hyper});
\item a double cover $X_2 \to \Gr(2,5)$ ramified in a quadratic section (Gushel--Mukai varieties, see~\cite{kuznetsov2014derived}).
\end{itemize}

In all these cases one can check that the category $\CA_X$ has the same Hochschild homology as the derived category of a K3 surface.
Moreover, for special cubic fourfolds the category $\CA_{X_3}$ is equivalent to $\BD^b(S)$ for a K3 surface $S$ (see~\cite{kuznetsov2010derived})
and the same is expected to be true for some Gushel--Mukai sixfolds (see~\cite{kuznetsov2015gushel}).
It is also expected that the same is true for special Debarre--Voisin varieties. Thus, it is natural to consider 
these categories as noncommutative K3 surfaces (or as K3 categories).

\begin{remark}\label{remark-gushel-mukai}
In the last example one can replace $\Gr(2,5)$ by its linear section $M$ of codimension $k \le 3$ and then
for odd $k$ take $X$ to be a quadric section of $M$ and for even $k$ take $X$ to be a double
covering of $M$ ramified in a quadric. In all these cases $\CA_X$ is a K3 category (\cite{kuznetsov2015gushel}).
\end{remark}

One of the interesting properties K3 surfaces have, is that moduli spaces of sheaves on them carry a symplectic structure,
and so when smooth and compact they are hyper-k\"ahler varieties. One can use K3 categories
in the same way. In fact, it was shown in~\cite{kuznetsov2009symplectic} that any moduli space of sheaves
on a cubic fourfold $X_3$ carries a closed 2-form, and if all the sheaves parameterized by this moduli space
are objects of the category $\CA_{X_3}$, then the 2-form is nondegenerate. The same argument can be applied
to any K3 category to show that a moduli space of objects in it carries a symplectic form. This allows
constructing new examples of hyper-k\"ahler varieties. In case of $X_3$ this gives the classical Beauville--Donagi
fourfold~\cite{beauville1985variete} or a more recent eightfold~\cite{lehn2015twisted}. Applied to $X_2$
this gives a double EPW sextic~\cite{iliev2009fano} and for $X_1$ presumably one can get the Debarre--Voisin
fourfold~\cite{debarre2010hyper}. Other moduli spaces and other examples of K3 categories may give new hyper-k\"ahler varieties.

However, finding other examples of noncommutative K3 categories
seems to be a difficult problem.
For instance, one can obtain a long list 
of hypersurfaces $X$ in weighted projective spaces with $\CA_X$ being a K3 category. But it looks as most of them 
are equivalent to derived categories of K3 surfaces, or reduce to one of the three above examples.

For instance, one can take a degree 4 hypersurface $X_4 \subset \PP(1,1,1,1,1,3)$. But the equation of $X_4$
after appropriate change of coordinates necessarily takes form
$x_5x_4 + f_4(x_0,\dots,x_4) = 0$.
Then $X_4$ can be obtained from $\PP^4$ by the blowup of the surface 
$S = \{ x_4 = f_4(x_0,\dots,x_4) = 0 \}$
followed by the contraction of the proper preimage of the hyperplane $\{x_4 = 0\}$. This allows to show
that $\CA_{X_4} \cong \BD(S)$.

\subsection{3-Calabi--Yau categories}\label{subsection-3cy}

As Calabi--Yau threefolds are of a special interest for physics, let us also list some examples of varieties, containing a 3-Calabi--Yau category:

\begin{itemize}
\item a cubic 7-fold $X_3 \subset \PP^8$;
\item an intersection of a quadric and a cubic $X_{2,3} \subset \PP^7$;
\item an intersection of $\Gr(2,6)$ and a quadric $X_{2} \subset \Gr(2,6)$;
\item a hyperplane section $X_1 \subset \Gr(3,11)$;
\item a hyperplane section $X'_1 \subset \Gr(4,9)$;
\item an intersection of $\SGr(3,6)$ with a quadric $X'_2 \subset \SGr(3,6)$;
\item an intersection of $\OGr_+(5,10)$ with a quadric $X''_2 \subset \OGr_+(5,10) \subset \PP^{15}$;
\item an intersection of $\PP^3\times\PP^3 \subset \PP^{15}$ with a quadric $X'''_2 \subset \PP^3 \times \PP^3$;
\item a double covering $X''''_2 \to \PP^5$ ramified in a quartic;
\item a double covering $X'''''_2 \to \GTGr$ ramified in a quadric.
\end{itemize}

\begin{remark}
Note that a quadric for $\OGr_+(5,10)$ in the spinor embedding corresponds to a Pl\"ucker hyperplane.
Also like in Remark~\ref{remark-gushel-mukai} one can take $M$ to be a general (spinor) linear section of $\OGr_+(5,10)$ 
of codimension $k \le 5$ and then for even $k$ take $X$ to be a quadric section of $M$, and for odd $k$ take $X$ 
to be the double covering of $M$ ramified in a quadric. In all cases we will get a 3CY category (this is analogous 
to Gushel--Mukai varieties). Similarly, one can take $M$ to be a general hyperplane section of $\SGr(3,6)$ or $\PP^3\times\PP^3$
and take $X$ to be the double covering of $M$ ramified in a quadric.
\end{remark}

\begin{remark}
In~\cite{iliev2011fano} there are other examples of Hodge-theoretic 3CY Fano varieties.
For most of these varieties $X$ there is indeed a semiorthogonal component $\CA_X \subset \BD(X)$
which is 3-Calabi--Yau, but it is equivalent to $\BD(Y)$ for a certain Calabi--Yau 3-fold.
In the following table we list in the left column such Fano varieties and in the right column
the corresponding Calabi--Yau 3-folds.
\begin{equation*}
\begin{array}{|c|c|}
\hline
\text{Variety from~\cite{iliev2011fano}} & \text{Corresponding CY $3$-fold} \\
\hline
(\PP^1)^6 \cap H & (\PP^1)^5 \cap H_1 \cap H_2 \\
\hline
((\PP^1)^3 \times \PP^3) \cap H & (\PP^1 \times \PP^1 \times \PP^3) \cap H_1 \cap H_2 \\
\hline
(\PP^2)^4 \cap H & (\PP^2)^3 \cap H_1 \cap H_2 \cap H_3 \\
\hline
(\PP^4)^3 \cap H & (\PP^4)^2 \cap H_1 \cap H_2 \cap H_3 \cap H_4 \cap H_5 \\
\hline
(\Gr(2,5) \times \Gr(2,5)) \cap H & \Gr(2,5) \cap \Gr(2,5) \\
\hline
\end{array}
\end{equation*}
Here $H$ and $H_i$ denote general hyperplanes in a natural embedding --- in the first two lines with respect to one half of the anticanonical divisor,
in the third line with respect to one third of the anticanonical divisor, and in the last two lines with respect to one fifth of the anticanonical divisor. 
The two Grassmannians in the last cell of the table are considered as embedded into the same $\PP^9$ but in a different way (in other words,
the second Grassmannian is the image of the first Grassmannian under a general element of the group $\PGL(10)$ acting naturally on $\PP^9$).

Similarly, the category appearing in $\BD(X'''_2)$ (a quadric section of $\PP^3\times\PP^3$) can be shown to be equivalent
to the twisted derived category of a small resolution of singularities of a special octic double solid (which is a Calabi--Yau 3-fold).
\end{remark}

\begin{remark}
Other examples from~\cite[Thm.~4.4]{iliev2011fano} can be explained by homological projective duality (see~\cite{kuznetsov2007homological,kuznetsov2014semiorthogonal}). 
It seems that the homological projective dual
for $\OO\PP^2$ is the Cartan cubic in $\PP^{26}$,
for $\SS_{12}$ is the double covering of $\PP^{31}$ ramified in the Igusa quartic,
for $\Gr(2,10)$ is the Pfaffian quintic, and
for $\SS_{14}$ is the double covering of $\PP^{63}$  ramified in the Popov octic. Then by HPD the nontrivial components of their linear sections
are equivalent to the nontrivial components of the corresponding linear sections of their dual varieties. Thus the examples in~\cite[Thm.~4.4]{iliev2011fano}
should reduce to a 7-dimensional cubic, a quartic double $\PP^5$, a quintic in $\PP^4$, and octic double $\PP^3$ respectively.
The first two of them are in our list, and the last two are Calabi--Yau 3-folds.
\end{remark}

The only example in~\cite{iliev2011fano} not covered by our approach is the double cover of $\IGr(2,6)$ ramified in a quadric.
It would be interesting to find out, whether or not it has a 3-Calabi--Yau subcategory. For this the results mentioned in Remark~\ref{remark-non-rectangular}
might be useful.

\section{Calabi--Yau categories}\label{section-properties}

\subsection{Indecomposability}\label{subsection-indecomposability}

One of the fundamental properties of Calabi--Yau categories is indecomposability.
It can be proved by a simple generalization of the beautiful argument of Bridgeland~\cite{bridgeland1999equivalences}.

Recall that a triangulated category $\CT$ is called {\sf connected} if $\HOH^0(\CT) = \kk$
(if $\CT = \BD(X)$ then by Hochschild--Kostant--Rosenberg isomorphism one has $\HOH^0(\CT) = H^0(X,\CO_X)$,
so $\CT$ is connected if and only if $X$ is).

\begin{proposition}
If $\CT \subset \BD(X)$ is a Calabi--Yau admissible subcategory then any semiorthogonal decomposition of $\CT$ is completely orthogonal.
In particular, if $\CT$ is connected then $\CT$ is indecomposable.
\end{proposition}
\begin{proof}
Assume $\CT = \langle \CA,\CB \rangle$ is a semiorthogonal decomposition. Then for any $A \in \CA$, $B \in \CB$ we have
\begin{equation*}
\Hom(A,B) = 
\Hom(B,\SSS_\CT(A))^\vee =
\Hom(B,A[n])^\vee = 0,
\end{equation*}
since $A[n] \in \CA$ as $\CA$ is triangulated. Thus the decomposition is completely orthogonal.
Therefore by Lemma~\ref{lemma-choh-additivity} we have $\HOH^\bullet(\CT) = \HOH^\bullet(\CA) \oplus \HOH^\bullet(\CB)$.
So, if $\CT$ is connected it follows that either $\HOH^0(\CA) = 0$ or $\HOH^0(\CB) = 0$. But then again 
by Lemma~\ref{lemma-choh-additivity} we have $\CA = 0$ or $\CB = 0$.
\end{proof}

Because of their indecomposability connected Calabi--Yau categories can be considered as the simplest building blocks 
of geometric triangulated categories.

\subsection{Hochschild homology}\label{subsection-properties}

Let $\CA \subset \BD(X)$ be an admissible subcategory. Then it extends in two ways to a semiorthogonal decomposition
\begin{equation*}
\BD(X) = \langle \CA^\perp, \CA \rangle
\qquad\text{and}\qquad
\BD(X) = \langle \CA, {}^\perp\CA \rangle.
\end{equation*}
Denote by $P^R_\CA, P^L_\CA \in \BD(X\times X)$ the kernels of the projections onto $\CA$ with respect to these decompositions.

\begin{lemma}\label{cylr}
If $\CA$ is $n$-Calabi--Yau then there exists a canonical isomorphism $P^L_\CA[n] \cong P^R_\CA \circ \SSS_X$.
\end{lemma}
\begin{proof}
Let $\alpha:\CA \to \BD(X)$ be the embedding functor. Then $P^L_\CA = \alpha\alpha^*$ and $P^R_\CA = \alpha\alpha^!$. 
By Lemma~\ref{lemma-serre-basics}(i) we have
\begin{equation*}
\alpha^! \circ \SSS_X = \SSS_{\CA} \circ \alpha^* = \alpha^*[n].
\end{equation*}
Composing this with $\alpha$ we deduce the claim.
\end{proof}

It follows that for a Calabi--Yau subcategory the Hochschild cohomology coincides with
the Hochschild homology, up to a shift.

\begin{proposition}\label{proposition-cy-hoh}
If $\CA \subset \BD(X)$ is an $n$-Calabi--Yau admissible subcategory
then
$$
\HOH^k(\CA) \cong \HOH_{k-n}(\CA)
$$
for any $k \in \ZZ$.
\end{proposition}
\begin{proof}
As before, let $P^L_\CA$ and $P^R_\CA$ be the kernels of the left and right projection onto $\CA$.
Let $\Delta:X \to X \times X$ be the diagonal embedding. By \cite[Prop.~8.1]{kuznetsov2009hochschild} we have
\begin{equation*}
\HOH^\bullet(\CA) = H^\bullet(X,\Delta^!P^L_\CA).
\end{equation*}
Also, the same argument shows that
\begin{equation*}
\HOH_\bullet(\CA) = H^\bullet(X,\Delta^*P^R_\CA).
\end{equation*}
Now using Lemma~\ref{cylr} we obtain
\begin{equation*}
H^\bullet(X,\Delta^!P^L_\CA) =
H^\bullet(X,\Delta^!(P^R_\CA \circ \SSS_X)[-n]) =
H^\bullet(X,\Delta^*P^R_\CA [-n])
\end{equation*}
which gives the required identification.
\end{proof}

The following is a useful consequence of this result.

\begin{corollary}\label{corollary-hoh-cy}
If $\CA \subset \BD(X)$ is a nonzero $n$-Calabi--Yau admissible subcategory
then $\HOH_{-n}(\CA) \ne 0$.
\end{corollary}
\begin{proof}
We have $\HOH_{-n}(\CA) = \HOH^0(\CA)$ which is nonzero by Lemma~\ref{lemma-choh-additivity}.
\end{proof}

\begin{remark}
The Proposition can be also used to show that some Calabi--Yau categories of geometric origin are connected (and hence indecomposable). For example, in the situation
of Corollary~\ref{corollary-divisor-pn} by Proposition~\ref{proposition-cy-hoh} we have
$\HOH^0(\CA_X) = \HOH_{-(n+1)(d-2)/d}(\CA_X)$. But by additivity Lemma~\ref{lemma-hoh-additivity}
and the HKR isomorphism this is equal to $H^{n-(n+1)/d,(n+1)/d-1}(X)$ which is one-dimensional by the Griffiths Residue Theorem.
\end{remark}

\subsection{The dimension of Calabi-Yau subcategories}\label{subsection-cydim}

Let $\CA \subset \BD(X)$ be a Calabi--Yau subcategory. In this section we discuss what can be said about the CY-dimension of $\CA$
and its relation to $\dim X$.

First, there are some reasons to believe in the following

\begin{conjecture}\label{congecture-cy-nonnegative}
If $\CA \subset \BD(X)$ is a Calabi--Yau subcategory then its CY-dimension is nonnegative.
\end{conjecture}

This Conjecture gives a lower bound for the CY-dimension of $\CA$. On the other hand, there is an evident upper bound.

\begin{theorem}\label{theorem-cy-dimension}
If $\CA \subset \BD(X)$ is an $n$-Calabi--Yau subcategory then $n \le \dim X$.
\end{theorem}
\begin{proof}
By Lemma~\ref{lemma-hoh-additivity} and Corollary~\ref{corollary-hoh-cy} we have $\HOH_{-n}(\BD(X)) \ne 0$.
But by Hochschild--Kostant--Rosenberg isomorphism we have
\begin{equation*}
\HOH_{-n}(\BD(X)) = \bigoplus_{p \in \ZZ} H^{p-n}(X,\Omega^p_X),
\end{equation*}
and if $n > \dim X$ the right hand side is zero.
\end{proof}

We conjecture that the inequality is strict unless $X$ itself is a blowup of a CY-variety.

\begin{conjecture}
Assume $\CA \subset \BD(X)$ is an $n$-Calabi--Yau category with
$n = \dim X$. 
Then there is a regular birational morphism $X \to X'$ onto a smooth projective variety $X'$
with trivial canonical class and an equivalence $\CA \cong \BD(X')$.
\end{conjecture}

This conjecture is clearly true in dimension 1 since all semiorthogonal components
of curves have been classified (see~\cite{okawa2011semi}). However, even in dimension 2
it is not quite clear.

Another interesting problem in this direction is the classification of Calabi--Yau categories of small dimension (say of dimension 0 and 1).
Also the question of boundedness (is there only a finite number of deformation families of CY categories of a given dimension?) is interesting.

One can also ask similar questions about fractional Calabi--Yau categories. For sure, they don't have the indecomposability property.

\begin{example}
Let $X \subset \PP^3$ be a cubic surface and $\CA = \CO_X^\perp \subset \BD(X)$.
Then $\CA$ is a connected fractional Calabi--Yau category of dimension $4/3$. However,
if $X$ can be represented as a blowup of $\PP^2$ in 6 points (e.g., if the base field
$\kk$ is algebraically closed of characteristic zero) then $\CA$ is generated
by an exceptional collection, so it is far from being indecomposable.
\end{example}

It seems, however, quite plausible that analogues of Theorem~\ref{theorem-cy-dimension}
and of Conjecture~\ref{congecture-cy-nonnegative} could be true for fractional Calabi--Yau categories.


\end{document}